\documentclass[12pt,a4paper]{article}
\usepackage{fullpage}
\usepackage[english]{babel}
\usepackage{amsmath}
\usepackage{amsthm}
\usepackage{amssymb}
\usepackage{graphicx}
\usepackage{tikz,pgfplots,pgfplotstable}
\usepackage{color}
\usepackage{hyperref}             
\hypersetup{linktocpage}
\usepackage{versions}
\usepackage{frcursive}
\usepackage{mathrsfs}
\usepackage{mathtools,stmaryrd}
  \newcommand{\nm}{\noalign{\smallskip}}

\newtheorem{definition}{Definition}[section]
\newtheorem{theorem}{Theorem}
\newtheorem{proposition}{Proposition}[section]

\newtheorem{lemma}[proposition]{Lemma}
\newtheorem{remark}{Remark}[section]
\numberwithin{figure}{section}

\DeclareMathOperator{\Span}{Span}

\newcommand{\cC}{{\cal C}}
\newcommand{\cD}{{\cal D}}

\newcommand{\cH}{{\cal H}}

\newcommand{\cP}{{\cal P}}

\newcommand{\T}{{\mathcal{T}}}

\newcommand{\sym}{{\text{sym}}}
\newcommand{\pw}{{\text{pw}}}
\newcommand{\td}{{\text d}}

\newcommand{\e}{{\varepsilon}}
\newcommand{\g}[1]{\mathbf{#1}}

\newcommand{\loc}{\text{loc}}

\newcommand{\norm}[2]{\left\|{#1}\right\|_{#2}}

\newcommand{\A}{\mathbb{A}}
\newcommand{\F}{\mathbb{F}}
\newcommand{\M}{\mathbb{M}}
\newcommand{\R}{\mathbb{R}}

\newcommand{\1}{{\bf1}}

\newcommand{\lra}{{\longrightarrow}}

\newcommand{\tin}{\text{in} \ }

\makeindex

\author{Habib Ammari \thanks{\footnotesize Department of Mathematics, 
ETH Z\"urich, 
R\"amistrasse 101, CH-8092 Z\"urich, Switzerland (habib.ammari@math.ethz.ch).}\and  Elie Bretin\thanks{\footnotesize Institut Camille Jordan, INSA de Lyon \& UCBL, Lyon, F-69003, France (elie.bretin@insa-lyon.fr). } \and  Pierre Millien\thanks{\footnotesize  Institut Langevin,  1 Rue Jussieu, 75005 Paris, France (pierre.millien@espci.fr).}\and Laurent Seppecher \thanks{\footnotesize  Institut Camille Jordan, Ecole Centrale de Lyon \& UCBL, Lyon, F-69003, France (laurent.seppecher@ec-lyon.fr). }}

\title{A direct linear inversion for discontinuous elastic parameters recovery from internal displacement information only}

\begin{document}

\date{}

\maketitle 
\begin{abstract}
The aim of this paper is to present and analyze a new direct method for solving the linear elasticity inverse problem. Given measurements of some displacement fields inside a medium, we show that a stable reconstruction of elastic parameters is possible, even for discontinuous parameters and without boundary information. We provide a general approach based on the weak definition of the stiffness-to-force operator which conduces to see the problem as a linear system. We prove that in the case of shear modulus reconstruction, we have an $L^2$-stability with only one measurement under minimal smoothness assumptions. This stability result is obtained though the proof that the linear operator to invert has closed range. We then describe a direct discretization which provides stable reconstructions of both isotropic and anisotropic stiffness tensors.

\end{abstract}

\def\keywords2{\vspace{.5em}{\textbf{  Mathematics Subject Classification
(MSC2000).}~\,\relax}}
\def\endkeywords2{\par}
\keywords2{35R30, 35B35, 65N21.}

\def\keywords{\vspace{.5em}{\textbf{ Keywords.}~\,\relax}}
\def\endkeywords{\par}
\keywords{Elastography, Inverse Problem, Shear Modulus Imaging.}
\tableofcontents

\section{Introduction}

Elastography is an imaging modality that aims at reconstructing the mechanical properties of tissues. The local values of  the elastic parameters can be used as a discriminatory criterion to differentiate healthy tissues from diseased tissues \cite{sarvazyan1995biophysical}.  Elasticity imaging emerged in the late $80$'s and early $90$'s as a way to improve the diagnostics on ultrasound images \cite{lerner1988sono}. 
A variety of techniques have been developed since then to assess the elastic parameters of tissues in vivo. For a comprehensive list of the different seminal works on the subject, we refer the reader to the reviews \cite{gennisson2013ultrasound,parker2010imaging,doyley2012model,wang2015optical}.

Most of the elastography methods are based on the following four steps:
\begin{enumerate}
\item[(i)] Perturb a medium with a mechanical stimulation (static, harmonic, or transient);
\item[(ii)] Image the deformation of the medium (usually via ultrasound imaging, magnetic resonance imaging, or optical coherence tomography);
\item[(iii)] Reconstruct the displacement field or some of its components in the medium;
\item[(iv)] Reconstruct the mechanical properties of the medium by solving an inverse problem.
\end{enumerate}

In most cases, the scale of the imaging resolution and the amplitude of the displacement field justify the use of linear elasticity model:

\begin{equation}\nonumber
-\nabla\cdot (\g C:\nabla^s\g u) = \g f,
\end{equation}
where $\g C = \{\g C_{ijkl}\}_{1\leq i,j,k,l \leq d}$ is the order four unknown elasticity tensor in dimension $d\in\{2,3\}$, $\nabla^s\g u:=(\nabla\g u+\nabla\g u^T)/2$ is the strain tensor associated to the displacement field $\g u$. The internal force density $\g f$ depends on the type of  source excitation: In the elastostatic regime, $\g f$ is zero, in elastodynamics, $\g f=\partial_{tt}\g u$ or $\g f=-\omega^2\g u$ in the time harmonic regime.

\subsection{Scientific context}

We consider the problem of reconstructing the elasticity tensor $\g C$
from the knowledge of a finite number $n$ of displacement fields $\{{\g u^{\ell} }\}_{\ell =1}^n$ solutions of the system of linear elasticity 
$$ - \nabla \cdot(\g C : \nabla^s \g u^{\ell}) =\g f^{\ell},$$
where the force densities $\g f^{\ell}$ are assumed to be known. In the isotropic elastic cases, the tensor $\g C$ can be written as 
$$\g C =  2 \mu {\g I}+ \lambda I \otimes I,$$
where $\mu$ and $\lambda$ are the Lam\'e coefficients, $\g I$ is the identity tensor $\g I_{ijkl}=\delta_{ik}\delta_{jl}$ and $I$ is the identity matrix $I_{ij}=\delta_{ij}$. Note that some results about the stability of this inverse problem can be found in \cite{ammari2015stability, widlak2015stability,bal2015reconstruction}. 

Before reviewing the different inversion methods already developed for the fourth step, it is important to have in mind the methods available for the reconstruction of the displacement field (third step). Displacement field reconstructions methods fall into two categories:
\begin{itemize}
\item[(i)] Methods that, given images of the unperturbed and the perturbed medium, use a mathematical treatment to recover the geometrical transformation between the images. Such methods can be based, for exemple, on speckle correlation technique \cite{thielicke2014pivlab}, optimal control \cite{ammaribretin2015mathematical} or optimal transport \cite{haker2004optimal}.
\item[(ii)] Direct reconstruction of the displacement field (or one of its components) during the imaging procedure.  Since ultrasound and OCT are imaging modalities that rely on the computation of a \emph{travel time} in a single scattering regime, axial displacements that are one or two orders of magnitude below the resolution of the imaging modality can be directly reconstructed by measuring a phase shift of the backscattered echo, with a very high frame rate ($\sim 10$ KHz for ultrasound  \cite{sandrin2002shear}, $\sim 700$ Hz for OCT \cite{wang2006tissue,nahas2013supersonic}). Although generally only the axial displacement is recovered by this method, a smart illumination sequence allows for a recovery of the axial and the lateral displacements \cite{tanter2002ultrafast,bal2015displacement}. Step $2$ and $3$ are therefore performed simultaneously.
\end{itemize}

A variety of methods are already available to perform the recovery problem (fourth step), depending on the type of mechanical stimulation, 
the data available (full internal displacement field or partial displacement field, single or multiple measurements), 
or the used  model (linear compressible elasticity or incompressible Navier equation). Most inversion algorithms roughly fall into one of these categories:
\begin{itemize}
\item[(i)] Resolution of a first-order transport equation \cite{ji2003recovery,mclaughlin2003unique,mclaughlin2010calculating,bal2014reconstruction};
\item[(ii)] Algebraic inversions \cite{barbone2007elastic,sandrin2002shear,bercoff2004supersonic,bal2015reconstruction};
\item[(iii)] Iterative inversions \cite{ammari2015mathematical,ammari2008method,ammari2015viscoelastic}.
\end{itemize}

First order methods and algebraic inversions are stable under some regularity assumptions on the elastic parameters of
the medium and the reconstructed displacement field, but their performances decrease when the Lam\'e parameters 
are not locally differentiable, which is often the case in biological media, or when the reconstruction of the displacement field is noisy. Moreover, they assume boundary knowledge which is usually not available in biomedical applications.

Iterative inversions assume less regularity for the elastic parameters,
but are computationally more costly, since a forward problem needs to be simulated at each step. 
In practice, it is difficult to use this approach  because some boundary information is required.

In clinical applications, the current state of the art for ultrasound and OCT based shear modulus imaging 
\cite{sandrin2003transient,bercoff2003vivo,chauvet2016vivo,nahas2013supersonic} is the algebraic inversion method developed 
in \cite{bercoff2004supersonic,montaldo2009coherent}. It relies on the assumption that the medium is \emph{locally homogeneous} 
(but not necessarily isotropic \cite{gennisson2003transient}) and is based on the computation of the group speed of a shear wave. In this \emph{locally homogeneous} case, the different polarizations of the elastic waves are decoupled, and only one component of the displacement field is required to compute the shear modulus. More precisely, the shear wave equation is

\begin{equation}\nonumber
-\nabla\cdot(\mu\nabla^s\g u) = -\partial_{tt}\g u
\end{equation}
and if $\mu$ is constant almost everywhere, one can assume that 

\begin{equation}\label{eq:approx}
-\mu\nabla\cdot(\nabla^s\g u) \approx -\partial_{tt}\g u\quad\text{almost everywhere}.
\end{equation}
This approximation is  in general false but it allows to simply approach $\mu$ as the square of the group speed of the shear wave. One just has to observe shear waves displacement using a fast enough imaging method. Another possible technique is to directly use the approximation 

\begin{equation}\label{eq:inverse}
\mu\approx \frac{|\partial_{tt}\g u|}{|\nabla\cdot(\nabla^s\g u)|}
\end{equation}
at positions and times such that $\nabla\cdot(\nabla^s\g u)$ does not vanish. These methods have the advantages of being able to reconstruct a good image of the shear modulus from small sub-wavelength displacement fields, at a very low computational cost (no matrix inversion needed).
Nevertheless, the method fails to quantitatively reconstruct the shear modulus where the medium exhibits discontinuities or strong variations. It is also not applicable to elastostatic experiments as the term $\partial_{tt}\g u$ must not vanish.

The method that we propose in this article is directly inspired by the previous formulae. If one defines the linear operator $A_{\g u}:\mu\mapsto -\nabla\cdot(\mu\nabla^s\g u)$, the approximation made in \eqref{eq:approx} is in fact a diagonal approximation of $A_{\g u}$ defining the diagonal operator $D_{\g u}=\mu\mapsto -\mu\nabla\cdot(\nabla^s\g u)$ and the inverse formula \eqref{eq:inverse} is equivalent to $D_{\g u}(\mu)\approx -\partial_{tt}\g u$.

As we can imagine, approaching $A_{\g u}$ by $D_{\g u}$ can be very optimistic in strongly heterogeneous media. In this article, we directly study the operator $A_{\g u}$ in order to stably invert it  when it is possible.

\subsection{Outline of the article and the main results}

In this paper, we study a new direct inversion method for reconstructing coefficients of the elasticity tensor from internal fields measurements.
The outline of the paper is the following:
\begin{itemize}
\item[(i)] we introduce a general weak formulation for the inverse problem (Section \ref{sec:inverseweak});
\item[(ii)] we theoretically study the operator to invert (null space, closed range property, stability of the inversion) in the isotropic shear modulus imaging case (Section \ref{sec:sheartoforce});
\item[(iii)] we study the numerical performance of the method  in the previous case as well as in some more general frameworks, in particular,  for reconstructing both Lam\'e coefficients and anisotropic media (Section \ref{sec:num}).
\end{itemize}

The strength of this direct inversion method is the fact that it combines the low computational cost of an algebraic inversion method (compared to the costly iterative methods) without requiring the high regularity assumptions on the coefficients to be reconstructed.

The determination of the null space of the operator to invert (Theorem \ref{theo:sbv}) and the main stability result (Theorem \ref{them:stability}) for the inversion are obtained under a weak regularity assumption on the coefficients of  the elasticity tensor, allowing the reconstruction of discontinuous coefficients. To the best of our knowledge, it is the first time that a non-iterative inversion method is theoretically studied for discontinuous elastic coefficients. The main consequence of this theoretical study is that the stable reconstruction of a discontinuous shear modulus is possible from one single measurement.

The numerical experiments shown in Section \ref{sec:num} are also new. We introduce a $\cP^1-\cP^0$ finite elements basis for the resolution of the inverse problem. We show that the sharp reconstruction of discontinuous coefficients from a minimal number of measurements is possible, and therefore that the theoretical results of Section \ref{sec:sheartoforce}  numerically hold in a more general setting.

%

\section{The inverse problem}\label{sec:inverseweak}

\subsection{The direct weak formulation}

Consider a smooth elastic medium $\widetilde \Omega\subset\R^d$, $d=2$ or $3$ with linear elastic properties described by the unknown elasticity tensor $\g C(x)\in T^4_\sym$. The space $T^4_\sym$ as well as the different tensor products are all defined in Definition \ref{de:tensornotation}. We assume that the unknown tensor $\g C$ belongs to $L^\infty(\widetilde \Omega,T^4_\sym)$. Consider now that one has measured internal displacement field $\g u$, which corresponds to the internal force density $\g f$  in some smooth subdomain of interest $\Omega\subset\widetilde\Omega$. The field $\g u \in H^1(\widetilde \Omega,\R^d)$ satisfies the linear elasticity equation 

\begin{equation*}
  - \nabla \cdot(\g C : \nabla^s {\g u}) ={\g f}^{} \quad \text{ in } \Omega,
\end{equation*}
in the sense of distributions, i.e., in $\cD'(\Omega,\R^d)$. In the case of multiple measurements, we assume knowledge of a finite number $n$ of force densities $\g f^{\ell}$ and the corresponding displacement fields $\g u^{\ell}$ satisfying linear elasticity equation

\begin{equation*}
  - \nabla \cdot(\g C : \nabla^s \g u^{\ell}) =\g f^{\ell} \quad \text{ in } \Omega,  \label{eqn:PDE_u}
\end{equation*}
in the sense of distributions. As $\g C\in L^\infty(\Omega,T^4_\sym)$ and $\nabla^s\g u\in L^2(\Omega,\R^{d\times d}_\sym)$, the previous equation makes sense in $H^{-1}(\Omega,\R^d)$ writing

\begin{equation}\label{eq:vf}
\int_{\Omega} (\g C: \nabla^s \g u^{\ell}) :  \nabla^s {\g v}=  \left< \g f^{\ell},{\g v}\right>_{H^{-1},H^1_0} , \quad  \forall {\g v}\in H^1_0(\Omega,\R^d).
\end{equation}
Here, $\left< \; , \; \right>_{H^{-1},H^1_0}$ denotes the duality pairing between $H^{-1}$ and $H^1_0$. 
Note that, by considering this problem in $H^{-1}$ (taking test functions in $H^1_0$), we naturally forget what happens on the boundary. This classical weak formulation naturally introduces a bilinear form $a_{\g C}$ such that the forward problem reads

\begin{equation*}
a_{\g C}(\g u, \g v)= l(\g v), \quad \forall \g v \in H^1_0(\Omega,\R^d).
\end{equation*}

The method that we present is based on the simple idea of changing the point of view and, given a vector field $\g u$,  writing (\ref{eq:vf}) as a bilinear form acting on $(\g C, \g v)$ instead of $(\g u,\g v)$:

\begin{equation}\label{eq:vfC}
a_{\g u}(\g C, \g v) = l(\g v), \quad \forall \g v \in H^1_0(\Omega,\R^d).
\end{equation}
In order to stay in a Hilbert space framework,  we make the non-restrictive assumption that the strain tensor $\nabla^s\g u$ is bounded. We will stand under this hypothesis in the whole paper.

\begin{definition}[Stiffness-to-force operator] If $\nabla^s\g u\in L^\infty(\Omega,\R^{d\times d}_\sym)$, then \eqref{eq:vfC} canonically defines the bounded operator:

\begin{equation*}\begin{aligned}
 A_{\g u}:L^2(\Omega,T^4_\sym) &\lra H^{-1}(\Omega,\R^d)\\
\g C &\longmapsto -\nabla\cdot(\g C:\nabla^s \g u), 
\end{aligned}
\end{equation*}
which is called the stiffness-to-force operator.
\end{definition}

Hence, the general inverse problem that we want to solve simply reads $A_{\g u}\g C=\g 0$ in the elastostatic  case and $A_{\g u}\g C=\g f$ in the elastodynamic case. 

In most of the cases, we do not look for a general tensor  $\g C(x)\in T^4_\sym$ and we  know, a priori, that it can be decomposed as a sum of known directions:

\begin{equation*}
 \g C(x) =  \sum_{k=1}^{N}  \mu^{(k)}(x)\g C^{k}, \quad \forall x \in \Omega,
\end{equation*}
where  $\mu^{(k)}$ are {\bf unknown} functions of $L^2(\Omega)$ and $\g C^{k}$ are {\bf known} constant tensors.
For instance, in isotropic cases $\g C(x)=2\mu(x)\g I$ or $\g C(x)=2\mu(x)\g I+\lambda(x)I\otimes I$, where $\mu$ and $\lambda$ are the two Lam\'e parameters. Hence, the reconstruction of $\g C(x)$ can be obtained from the reconstruction of the $\mu^{(k)}(x)$  solutions of the variational problem:

\begin{equation}\nonumber
 \sum_{k=1}^{N} \int_{\Omega} \mu^{(k)}(x)  \big(\g C^{k}: \nabla^s \g u^{\ell}(x)\big) :
\nabla^s {\g v(x)}\td x =  \left< \g f^{\ell},{\g v}\right>_{H^{-1},H^1_0}\quad  \forall {\g v}\in H^1_0(\Omega,\R^d), \end{equation}
or equivalently

\begin{equation*}
\sum_{k=1}^{N} \left<A^{\g C^{k}}_{\g u^{\ell}}(\mu^{(k)}),{\g v}\right>_{H^{-1},H^1_0} =   \langle \g f^{\ell},{\g v} \rangle_{H^{-1},H^1_0}.
\end{equation*}
Here, for all ${\g u} \in W^{1,\infty}(\Omega,\R^d)$ and $\g C\in T^4_\sym$, the bounded linear operator $ A^{\g C}_{{\g u}} : L^2(\Omega) \rightarrow H^{-1}(\Omega, \R^d) $ is defined by

\begin{equation*}
A^{\g C}_{{\g u}}(\mu)= -\nabla\cdot(\mu\g C:\nabla^s\g u).
\end{equation*}
The general recovery problem with multiple measurements reads as the following system:
\begin{align*}
 \begin{pmatrix}
  A^{\g C^{1}}_{{\g u}^{1}} & \ldots &  A^{\g C^{N}}_{{\g u}^{1}} \\
  \vdots &  & \vdots \\
   A^{\g C^{1}}_{{\g u}^{n}} & \ldots &  A^{\g C^{N}}_{{\g u}^{n}} 
 \end{pmatrix} \begin{pmatrix} \mu^{(1)} \\ \vdots \\ \mu^{N)}
  \end{pmatrix} = \begin{pmatrix} \g f^{1} \\ \vdots \\ \g f^{n}
  \end{pmatrix}.
\end{align*}

\begin{remark}
 As we will see in Section \ref{sec:num}, this formulation is naturally adapted to a finite element discretization when looking for the coefficients 
 $\mu^{(k)}(x)$ in $\mathcal{P}^0$ using test functions in $\mathcal{P}^1_0$.
\end{remark}

\subsection{Existing stability results}

Although the question of the injectivity   is very hard without extra regularity assumptions on $\g C$, 
there exists some stability results for the reconstruction of the tensor $\g C$. The most important one can be found in \cite{bal2015reconstruction}.
We include here, for the sake of completeness, the following stability result for data with $W^{2,\infty}$ regularity for $\g u$.

\begin{theorem} [see {\cite{bal2015reconstruction}}\label{theo:bal}] ~ Let $({\g u^{1},\dots,\g u^{n}})$ and $({\g {\tilde u}^{1},\dots,\g {\tilde u}^{n}})$ be two families of displacement fields of size
$n = d(d+1)/2 + N/d$
and $A:=(A_{\g u^{1}},\dots, A_{\g u^{n}})$,  $\tilde A:=(A_{\tilde{\g u}^{1}},\dots, A_{\tilde{\g u}^{n}})$ be the corresponding  multiple data stiffness-to-force operators. If the tensors  $\g C$ and $\widetilde{\g C}$ satisfy

\begin{equation}\nonumber
A\g C = \g 0\quad\text{ and }\quad\tilde A\widetilde {\g C} = \g 0,
\end{equation}
under some extra assumptions on the linear independence  of these families of displacement fields, then $\g C$ and $\widetilde{\g C}$ can each be uniquely reconstructed over $\Omega$ up to a multiplicative constant. Moreover, if we assume that $\norm{\g C}{L^\infty(\Omega)}=\|\widetilde {\g C}\|_{L^\infty(\Omega)} $, then 

\begin{equation}\nonumber
\|\g C - \widetilde{\g C} \|_{L^{\infty}(\Omega)} + \|\nabla \cdot \g C - \nabla\cdot\widetilde{\g C}\|_{L^{\infty}(\Omega)}  \leq k\sum_{\ell=1}^{n} \| \nabla^s \g u^{\ell} - \nabla^s \tilde{{\g u}}^{(\ell)} \|_{W^{1,\infty}(\Omega)},
\end{equation}
where $k$ does not depend on $\g C$ and $\widetilde{\g C}$.
\end{theorem}
We refer the reader to \cite{bal2015reconstruction} for more details.

\begin{remark} If we assume that the elasticity tensor $\g C$ is of the form  ${\g C}  = 2 \mu {\g I}$ 
or $\g C = 2 \mu {\g I} + \lambda  I \otimes I$, in dimension $2$, then 
Theorem \ref{theo:bal} implies that one needs at least $4$ sets of measurements in order to reconstruct $\g C$ up to a multiplicative constant. Moreover, one needs $\nabla^s\g u^\ell$ to be Lipschitz.
\end{remark}

\subsection{Classical elastic media inversion problems}

\subsubsection{Shear modulus inversion}

In the ideal case where $\lambda$ is equal to zero or it is assumed to be known in the medium, the elasticity equation reads as 
$$ - \nabla \cdot( \mu \nabla^s {\g u}^{\ell}) ={\g f}^{\ell} \quad\text{ in } H^{-1}(\Omega,\R^d)$$
which corresponds to the previous model with  $N=1$, $\mu = \mu^{(1)}$ and $\g C^{1} = {\g I}$.
Note that in the static case ($\g f^{\ell} = 0$), 
the recovery problem is equivalent to finding $\mu$ in the null space of $(A^{{\g I}}_{\g u})$. 
In particular, \emph{formally}, if $\g u^{\ell}$ is smooth and if $\nabla^s \g u^{\ell}$ 
is invertible, then 
$$ \nabla \cdot( \mu \nabla^s \g u^{\ell}) = 0 \ \mbox{ and }  \ \mu> 0 $$ implies that 
$$ \mu \nabla^s \g u^{\ell}
\left(\nabla\log(\mu) +   (\nabla^s \g u^{\ell})^{-1} \Delta^s\g u^{\ell}  \right) = 0,$$
which suggests that this equation has a non-trivial solution if and only if there  exists $\varphi$ such as 
$$(\nabla^s \g u^{\ell})^{-1} \Delta\g u^{\ell} = \nabla \varphi.$$
In that case,  
$$ N(  A^{{\g I}}_{\g u^{\ell}}) = \text{span}\left\{ \exp(-\varphi)\right\}.$$
The problem of showing that the null space of $ A^{{\g I}}_{\g u^{\ell}}$ is at most of dimension one
has been solved in  \cite{barbone2007elastic} in the case of \emph{smooth coefficients}. The  aim 
of the next section is to generalize this approach for \emph{discontinuous} strain tensors.  

\begin{remark} The method developed in \cite{barbone2007elastic} can be numerically implemented
by using the Helmholtz decomposition of  $(\nabla^s \g u^{\ell})^{-1} \Delta\g u^{\ell}$.   
Moreover, it suggests also that only one set of data is required to reconstruct $\mu$ up to a multiplicative constant.
\end{remark}

\subsubsection{Inversion of Lam\'e coefficients}

In the general isotropic case, i.e, $P=2$, $\mu^{(1)} = \mu$, $\g C^1 = {\g I}$, $\mu^{(2)} = \lambda$ and $\g C^2 = I \otimes I,$ 
the coefficient  $\lambda$ is associated to the operator $A^{I\otimes I}_{{\g u}}$ defined by 
$$ A^{I\otimes I}_{{\g u}}(\lambda) = \nabla \left( \lambda \nabla \cdot {\g u} \right).$$
Formally,  its null space is at most of dimension one and is given by    
$$ N\left( A^{I\otimes I}_{{\g u}}\right) =  \Span\left\{ \frac{1}{\nabla \cdot   {\g u}} \right\}.$$
In practice, this shows that the reconstruction of Lam\'e coefficients $(\mu,\lambda)$ requires at least two sets of data 
${\g u}^1$ and ${\g u}^2$: 
\begin{align*}
 \begin{pmatrix}
  2 A^{{\g I}}_{{\g u}^{1}} &   A^{I \otimes I}_{{\g u}^{1}} \\ 
  \nm
  2 A^{{\g I}}_{{\g u}^{2}} &   A^{I \otimes I}_{{\g u}^{2}}
 \end{pmatrix} 
 \begin{pmatrix}
 \mu \\ 
 \lambda
  \end{pmatrix} = 
  \begin{pmatrix} \g f^{(1)} \\ \nm  \g f^{(2)}
  \end{pmatrix},
\end{align*}
which satisfy the necessary condition  $ N\left( A^{I\otimes I}_{{\g u}^1}\right) \neq N\left( A^{I\otimes I}_{{\g u}^2}\right) $ or in other terms, 
$$ \forall c \in \R, \quad  \nabla \cdot {\g u}^1 \neq  c \nabla \cdot {\g u}^2.$$

\subsubsection{Anisotropic medium inversion}

The last example is an anisotropic medium such that the tensor takes 3 independent directions:
$${\g C(x)} =  \mu^{(1)}(x) {\g C}^{1} + \mu^{(2)}(x) {\g C}^{2} + \mu^{(3)}(x) {\g C}^{3},$$
where the tensors ${\g C}^{1}$, ${\g C}^{2}$ and ${\g C}^{3}$ are defined by  
  \begin{equation} \label{add1} {\g C}^{1} : A = \left( \begin{matrix}
                                    A_{11}  & 0 \\
                                   0 & 0
                                  \end{matrix} \right), \quad {\g C}^{2} : A = \left( \begin{matrix}
                                   0 & 0 \\
                                   0 & (A)_{22}
                                  \end{matrix} \right),
                 \end{equation}
 and 
  \begin{equation} \label{add2}  {\g C}^{3} : A = \left( \begin{matrix}
                                   0 & (A_{12} + A_{21})/2 \\
                                   (A_{12} + A_{21})/2  & 0
                                  \end{matrix} \right). \end{equation}
for any squared matrix $A$.

This is an ideal case, and doesn't necessarily correspond to a biomedical imaging application. It is used as a an example to show the versatility of our method. Anisotropic shear wave imaging is of great use in cardiac imaging. The anisotropic model for the myocardium and the imaging of the degree of anisotropy will be investigated in a forthcoming paper.

 \subsection{Regularity of the coefficients of elasticity tensor}                                  

 The choice of the functional spaces for the elasticity tensor's coefficients and for $\g u$ is a crucial question.
 The standard theory of elliptic systems shows that the regularity of $\g C$'s coefficients determines the regularity of the solution $\g u$ of the linear elasticity equation.  For instance, it is well-known that, under some ellipticity conditions, if the coefficients of $\g C$ are in $L^\infty$, the solution $\g u$ is in $H^1(\Omega,\R^2)$ and therefore no more than $L^2$ regularity can be expected for $\nabla^s \g u$.
 The standard H\"older theory for elliptic systems tells us that if the coefficients are piecewise H\"older continuous, then the same regularity can be expected for $\nabla^s \g u$. 
 
 As we mainly focus on imaging mechanical properties of biological tissues, we should use an appropriate model for the elastic coefficients. Typically, it is not realistic to assume that the elastic coefficients are everywhere differentiable, since biological tissues are often constituted of different types of embedded materials which exhibit discontinuities. 
 
 A good acceptable model for a biological medium is to assume that the biological parameters are piecewise smooth with smooth discontinuity surfaces. Out of these discontinuities,  we suppose a Sobolev type smoothness. We call such a space of function $W^{1,p}_{\pw}(\Omega)$, and give its precise definition in Definition \ref{de:A}.

 We will also use spaces that include discontinuous functions and that are more general than those in $W^{1,p}_{\pw}(\Omega)$. We introduce the subspace $SBV(\Omega) \subset BV(\Omega)$ of the functions of bounded variations whose discontinuity sets have no Cantor parts. The full precise definition is given in Definition \ref{de:SBV}.
 
 The relations between the functional spaces that we use are the following:
 \begin{align*}
 W^{1,p}(\Omega)\subset W^{1,p}_{\pw}(\Omega) \subset SBV^p(\Omega) \subset L^p(\Omega).
 \end{align*}
 
 \begin{remark} The condition that the coefficients of the tensor belong to $W^{1,p}_{\pw}(\Omega)$ or $SBV(\Omega)$ makes an important difference between this work and the aforementioned theoretical works on elastography. Under this assumption, we cannot assume that $\nabla^s \g u \in W^{1,p}$ and the analysis becomes more complicated. This is the reason for Section \ref{sec:sheartoforce} to be quite lengthy and technical.
 \end{remark}

\section{Shear modulus imaging: invertibility and stability in the isotropic case}\label{sec:sheartoforce}

In this section, we study the so-called \emph{shear-to-force operator} \begin{align*}
A^{\g I}_{\g u} : L^2(\Omega)&\longrightarrow H^{-1}(\Omega,\R^d)\\
 \mu & \longmapsto -\nabla \cdot \left(\mu \nabla^s \g u\right).
\end{align*} 
The outline of this section is the following:
\begin{enumerate}
\item[(i)] In Subsection \ref{sec:nullspace}, we study the null space of $A^{\g I}_{\g u}$ and we extend the results of \cite{barbone2007elastic} by showing that under low regularity assumptions for $\g u $ (typically, $SBV$ type regularity for $\nabla^s\g u$), the null space is of dimension zero or one.
\item[(ii)] In Subsection \ref{sec:closedrange}, we study the solvability of the inverse problem by giving sufficient conditions on $\nabla^s\g u$ for the operator $A^{\g I}_{\g u}$ to be of closed range, therefore ensuring the continuity of the inverse on the orthogonal of the null space. We first prove that the operator $A^{\g I}_{\g u}$ as closed range under invertibility and $W^{1,p}$ smoothness assumption for $\nabla^s\g u$, for some $p>d$. We then relax the regularity assumption to a piecewise regularity $\nabla^s\g u \in W^{1,p}_{\pw}(\Omega)$, ensuring the solvability of the inverse problem when looking for piecewise smooth shear modulii.
\item[(iii)] In Subsection \ref{sec:stability}, we give quantitative results on the stability of the inversion of $A^{\g I}_{\g u}$.
\end{enumerate}

\subsection{Spaces of discontinuous functions}

In order to prove invertibility and stability of the inverse problem under minimal smoothness assumptions on the coefficients and the data, we introduce here two spaces of discontinuous functions.

\subsubsection{The space $W^{1,p}_{\pw}(\Omega)$}

\begin{definition}\label{de:A}
A function $f$ is said to be in $W^{1,p}_{\pw}(\Omega)$ for $1\leq p\leq +\infty$, if there exists a smooth covering  $\Omega_1$,\ldots,$\Omega_k$, for $k\geq 1$,  such that
\begin{itemize}
\item[(i)]  $\Omega_i$ is a smooth open connected subdomain of $\Omega$ for every $i\in \{ 1,\ldots, k\}$;
\item[(ii)] $\Omega_i\cap \Omega_j=\emptyset$ if $i\neq j$;
\item[(iii)] $\displaystyle{\bigcup_{i=1}^k \overline{\Omega_i} = \overline\Omega};$
\item[(iv)] $\forall i \in \{ 1,\ldots, k\},\quad f\big\vert_{\Omega_i}\in W^{1,p}\left(\Omega_i\right).$
\end{itemize}
\end{definition}

In order to have an even more general set of discontinuous functions, we introduce the following space inspired by the space $SBV(\Omega)$.

\subsubsection{The space $\mathrm{SBV}^p(\Omega)$}

Since the derivative of a function $f\in BV(\Omega)$ can be decomposed as:
\begin{align*}
Df= \nabla f \mathcal{H}^d + [f]{\g n}_S \mathcal{H}^{d-1}_S + D_cf,
\end{align*}
where $\mathcal{H}^d$ is the Lebesgue measure on $\Omega$, $\mathcal{H}_S^{d-1}$ is the
surface Hausdorff measure on a rectifiable surface $S$, $\g n_S$ is a
normal vector defined almost everywhere on $S$, $ f\in L^1(\Omega)$ is the smooth
derivative of $f$, $[f]\in L^1(S,\mathcal{H}^{d-1}_S)$ is the jump of $f$ across $S$ and
$D_c f$ is a vector measure supported on a set of Hausdorff dimension less than $(d-1)$, which means that its ${(d-1)}$-Hausdorff-measure is zero. The well-known space $SBV(\Omega)$ introduced by De Giorgi and Ambrosio \cite{ambrosioSBV} is the subclass of $BV(\Omega)$ of functions whose derivative Cantor parts are zero: $D_cf=\g 0$. Following this idea, we introduce a very large piecewise-$W^{1,p}$ class of functions:

\begin{definition}\label{de:SBV} For  $1\leq p\leq +\infty$, we define
\begin{align*}
\mathrm{SBV}^p(\Omega)=\left\{f\in \mathrm{SBV}(\Omega)\cap L^p(\Omega),\ \nabla f\in L^p(\Omega,\R^d)\right\},
\end{align*}
where $\nabla f$ is the Lebesgue part of the measure $Df$. 
\end{definition}

Note that we clearly have the inclusion $W^{1,p}_{\pw}(\Omega)\subset SBV^p(\Omega)$.

\subsection{Null space of the shear-to-force operator} \label{sec:nullspace}

In this subsection, we prove Theorems \ref{theo:sobolev} and \ref{theo:sbv}, which give simple conditions on $S:=\nabla^S \g u$ in order to ensure that the operator $A^{\g I}_{\g u}$ has a null space of dimension zero or one.

\begin{theorem}[Characterization in $W^{1,p}$]\label{theo:sobolev} Assume that $\nabla^s\g u\in L^\infty(\Omega,\R^{d\times d}_\sym)\cap W^{1,p}(\Omega,\R^{d\times d})$ for some $p>d$ and that $|\det \nabla^s\g u|\geq c> 0$. Then the space
\begin{equation}\nonumber
K_{\g u}:=\left\{\mu\in L^2(\Omega),\ \nabla\cdot(\mu \nabla^s\g u)=0\right\},
\end{equation}
is of dimension zero or one. In the second case, there exists a positive continuous function $\mu_0$, such that $K_{\g u}=\text{span}\{ \mu_0\}$.
If $\Omega$ is Lipschitz, then $\mu_0$ belongs to $W^{1,p}(\Omega)$.
\end{theorem}

\begin{theorem}[Characterization in $SBV^{p}$] \label{theo:sbv} Assume that $\nabla^s\g u\in L^\infty(\Omega,\R^{d\times d}_\sym)\cap SBV^{p}(\Omega,\R^{d\times d})$ for some $p>d$ and that $|\det \nabla^s\g u|\geq c> 0$. Then, the space
\begin{equation}\nonumber
K_{\g u}:=\left\{\mu\in L^2(\Omega),\ \nabla\cdot(\mu \nabla^s\g u)=0\right\}
\end{equation}
is of dimension zero or one. 
\end{theorem}

\begin{proof} Denote $\Sigma$ the closure of the discontinuity surface of $S$. The open set $\Omega\backslash\Sigma$ can be decomposed as a countable union of connected open sets:

\begin{equation}\nonumber
\Omega\backslash\Sigma=\bigcup_{i\in I}\Omega_i.
\end{equation}
One may apply Theorem \ref{theo:sobolev} on each subset and say that there exists some $\nu_i\in\cC^0(\Omega_i)$ such that any solution of the problem is written as
\begin{equation}\nonumber
\mu=\sum_{i\in I}\alpha_ie^{\nu_i}\1_{\Omega_i}\quad\text{ in }\Omega\backslash\Sigma,
\end{equation}
where $\alpha_i$'s are some real numbers. 

We show now that these numbers are linked by the jump condition over $\Sigma$. Consider two subdomain $\Omega_i$ and $\Omega_j$ in contact in the sense that their common boundary $$\Sigma_{ij}:=\partial\Omega_i\cap\partial\Omega_j$$ is of positive surface measure: $\cH^{d-1}(\partial\Omega_i\cap\partial\Omega_j)>0$. As $\Sigma$ is rectifiable, there exists $x_0\in \Sigma_{ij}$ and $B:=B(x_0,\e)$ such that $\Omega_i^B:=\Omega_i\cap B$ and $\Omega_j^B:=\Omega_j\cap B$ are Lipschitz domains. As $\mu$ and $S$ are $W^{1,p}$ in $\Omega_i^B$ and $\Omega_j^B$, so is the product $\mu S$ and it admits two-sided traces $\mu_i S_i$ and $\mu_j S_j$ defined as functions of $L^p(\Sigma_{ij}\cap B)$. From the variational formulation, the jump condition at $\Sigma_{ij}\cap B$ reads as

\begin{equation}\nonumber
\mu_i S_i\nu = \mu_j S_j\nu \quad \text{ almost everywhere  on }\Sigma_{ij}\cap B.
\end{equation}
This jump condition gives a vectorial equation linking $\alpha_i$ and $\alpha_j$ which is 

\begin{equation}\label{eq_jump}
\alpha_i e^{\nu_i}S_i\nu = \alpha_j e^{\nu_j}S_j\nu.
\end{equation}
As $\nu_i$, $\nu_j$ are bounded in $B$ and $|\det S_i|$, $|\det S_j|\geq c>0$, there exists $c'>0$ such that $|e^{\nu_i}S_i\nu|\geq c'$ and $|e^{\nu_j}S_j\nu|\geq c'$. A first consequence is that if one $\alpha_i=0$ then they are all zero and $\mu=0$.

Now consider another solution $\mu'=\sum_{i\in I}\beta_ie^{\nu_i}\1_{\Omega_i}$ and assume that $\frac{\mu'}{\mu}$ is not constant. There exist $\Omega_i,\ \Omega_j$ in contact such that ${\beta_i}/{\alpha_i}\neq {\beta_j}/{\alpha_j}$. Using  \eqref{eq_jump} for both couples $(\alpha_i,\alpha_j)$ and $(\beta_i,\beta_j)$,  it follows that there exists $\gamma\neq 0$ such that $\alpha_j=\gamma\alpha_i$ and $\beta_j=\gamma\beta_i$, which leads to ${\beta_i}/{\alpha_i}= {\beta_j}/{\alpha_j}$. Since this is absurd, $\mu'/\mu$ is constant.
\end{proof}

\paragraph{Proof of Theorem \ref{theo:sobolev}}
 As $|\det S|\geq c> 0$, there exists $S^{-1}\in L^\infty(\Omega,\R^{d\times d}_\sym)$ such that $SS^{-1}=I$ almost everywhere in $\Omega$. We successively apply  Propositions \ref{prop_dec}  and  \ref{prop_carac} with $\g b=-S^{-1} \nabla\cdot S$, which ends the proof. \qed

\begin{proposition}[Decomposition]\label{prop_dec} Assume that $S\in L^\infty(\Omega,\R^{d\times d}_\sym)\cap W^{1,p}(\Omega,\R^{d\times d})$ for some $p\in[2,+\infty]$ and there exists $S^{-1}\in L^\infty (\Omega,\R ^{d\times d})$ such that $SS^{-1}=I$ almost everywhere on $\Omega$. Then, any solution $\mu$ of $\nabla \cdot (\mu S)=0$ is in $W^{1,1}(\Omega)$ and satisfies

\begin{equation}\nonumber
\nabla\mu +\mu S^{-1}(\nabla\cdot S) = {\g 0}.
\end{equation}
\end{proposition}

\begin{proof} As $(\nabla\cdot S)\in L^p(\Omega,\R^d)$, we have for every $\g v\in L^\infty(\Omega,\R^d)\cap W^{1,p}(\Omega,\R^d)$ that $\nabla\cdot (S\g v)\in L^p(\Omega)$ and $\nabla\cdot (S\g v)=(\nabla\cdot S)\cdot\g v+S:\nabla \g v$. Then, we write

\begin{equation}\nonumber
\int_\Omega \mu \nabla\cdot (S\g v) =\int_\Omega \mu (\nabla\cdot S)\cdot\g v.
\end{equation}

For any $\g w\in\cD(\Omega,\R^d)$, the test function $\g v=S^{-1}\g w$ belongs to $L^\infty(\Omega,\R^d)$ and $\nabla \g v= \nabla S^{-1}\cdot\g w + S^{-1}\cdot \nabla  \g w\in L^p(\Omega,\R^{d\times d})$, and can be used in the previous equation to get,

\begin{equation}\nonumber
\int_\Omega \mu \nabla\cdot \g w =\int_\Omega \mu (\nabla\cdot S)\cdot S^{-1} \g w,\quad \forall\g w\in\cD(\Omega,\R^d).
\end{equation}
This means

\begin{equation}\nonumber
\nabla\mu = -\mu S^{-1}(\nabla\cdot S).
\end{equation}
\end{proof}

\begin{proposition}[Regularity]\label{prop_reg}  Take $p>d$ and consider $\g b\in L^p(\Omega,\R^d)$. Any solution of

\begin{equation}\nonumber\left\{\begin{aligned}
\mu &\in L^2(\Omega),\\
\nabla\mu &=\mu\g b,
\end{aligned}\right.\end{equation}
belongs to $\cC^0(\Omega)$.  If $\Omega$ is Lipschitz, then it belongs to $W^{1,p}(\Omega)$.
\end{proposition}

\begin{proof} Consider a ball $B\subset\Omega$. As $p>d$, the injection $W^{1,p}(B)\hookrightarrow\cC^0(B)$ holds. Let us prove that $\mu\in W^{1,p}(B)$. First, note that as $p\geq 2$, $\nabla\mu=\mu\g b\in L^1(B,\R^d)$ and so $\mu\in W^{1,1}(B)$. Call now $q^*=\sup\{q\geq 1,\ \mu\in W^{1,q}(B)\}$. 

Suppose that $q^*\leq d$. For any $1\leq q<q^*$, $\mu\in W^{1,q}(B)\hookrightarrow L^r(B)$ with $\frac 1r=\frac 1q-\frac1d$ and $\mu\g b$ belongs to $L^s(B,\R^d)$ with $\frac 1s=\frac1r+\frac1p=\frac 1q-\frac1d+\frac1p$. Let $\beta=\frac1d-\frac1p>0$. We get that $\mu\in W^{1,s}(B)$ with $\frac1s = \frac1q-\beta$. One can choose $q$ such that $s>q^*$ which contradicts the definition of $q^*$. Then $q^*>d$.

Considering that $\mu\in W^{1,q}(B)$ for some $q\in(d,q^*]$ and that for such $q$, $W^{1,q}(B)\hookrightarrow L^\infty(B)$, we get that $\mu\g b\in L^p(B,\R^d)$ and so $\mu\in W^{1,p}(B)$. As a consequence, $\mu$ is continuous in $\Omega$.

If $\Omega$ is Lipschitz, one can restart the proof replacing $B$ by $\Omega$ to obtain that $\mu\in W^{1,p}(\Omega)$.

\end{proof}

\begin{proposition}[Existence of non-zero solutions]\label{prop_carac}  Take $p>d$ and consider $\g b\in L^p(\Omega,\R^d)$. The problem 

\begin{equation}\left\{\begin{aligned}\nonumber
&\mu\in L^2(\Omega),\\
&\nabla\mu=\mu\g b \quad \mbox{in }  \cD'(\Omega,\R^d),
\end{aligned}\right.\end{equation} 
admits a non-zero solution if and only if the vector field $\g b$ is conservative which means that

\begin{equation}\nonumber
\nabla\nu=\g b \quad  \mbox{in }  \cD'(\Omega,\R^d),
\end{equation}
 admits a continuous solution. In this case, the set of solutions is given by $\{\alpha e^\nu,\ \alpha\in\R\}$. Moreover, if $\Omega$ is Lipschitz, then $\nu$ is bounded and there exists a constant $m>0$ such that $\mu_0:=e^\nu\geq m$.
\end{proposition}

\begin{proof} If the equation $\nabla\nu=\g b$ admits a solution $\nu\in \cC^0(\Omega)$, then $e^\nu$ is continuous and positive in $\Omega$. It satisfies in the weak sense $\nabla(e^\nu)=e^\nu\g b$. Its inverse $e^{-\nu}$ has the same properties. Take $\mu\in L^2(\Omega)$ a solution of $\nabla\mu=\mu\g b$ and define $\alpha = \mu e^{-\nu}\in L^2_\loc(\Omega)$, then in the weak sense, $\nabla\alpha = e^{-\nu}\nabla\mu-\mu e^{-\nu}\g b = 0$, so, as $\Omega$ is connected, $\alpha$ is constant in $\Omega$. This proves the first part of the result.

Assume now that $\nabla\nu=\g b$ has no solution in $\cC^0(\Omega)$ and consider a solution $\mu\in L^2(\Omega)$ of the equation $\nabla\mu=\mu\g b$. Using Proposition \ref{prop_reg}, it follows that $\mu\in \cC^0(\Omega)$.

Suppose that $\mu$ does not vanish  in $\Omega$, then $\mu>0$ in $\Omega$ (take $-\mu$ if $\mu<0$), then $\nu:=\ln \mu$ is continuous and satisfies $\nabla\nu = \g b$, which is impossible.

As a consequence, $\mu$ does vanish somewhere in $\Omega$. If $\mu\neq 0$, then there exists a ball $B\subset\Omega$ such that $\mu>0$ (take $-\mu$ if $\mu<0$) in $B$ and $\mu$ vanishes somewhere on $\partial B$. Inside $B$, $\nu:=\ln\mu$ is continuous and satisfies $\nabla\nu=\g b\in L^p(B,\R^d)$ so $\nu\in W^{1,p}(B)\hookrightarrow L^\infty(B)$.  Thus,  $\mu=e^\nu\geq e^{-\norm{\nu}{L^\infty(B)}}>0$ on $B$, which  contradicts the fact that $\mu$ vanishes somewhere on $\partial B$. Finally, it follows that $\mu=0$ is the only solution. \end{proof}

\subsection{Closed range property of the shear-to-force operator}
\label{sec:closedrange}

In the case of existence of a non-trivial null space for the shear-to-force operator (elastostatic case), we study the possibility of a stable reconstruction of the parameter $\mu$ in $L^2(\Omega)$. We recall that a linear operator $A:H\rightarrow E$ where $H$ is an Hilbert space and $E$ a Banach space has closed range if $R(A):=A(H)$ is closed in $E$. The following proposition gives an equivalent definition of this property:

\begin{proposition}[Closed range operator] Let $H$ and $E$ be en Hilbert and a Banach space. A linear operator $A:H\rightarrow E$ has closed range if and only if there exists a constant $c>0$ such that 

\begin{equation}\nonumber
\forall x\in N(A)^\perp,\quad \norm{x}{H}\leq c\norm{Ax}{E}.
\end{equation}
\end{proposition}
In particular, this implies the existence of a bounded inverse operator from $R(A)$ to $N(A)^\perp$. We refer the reader to \cite[ Section 2.7]{brezis2010functional} for more details.

\begin{theorem}[Closed range with strain in $W^{1,p}$]\label{theo:closedrangew1p}
Take $\g u$ such that $S:=\nabla^s \g u \in W^{1,p}(\Omega)$, and that $\vert \text{det}\ \nabla^s \g u\vert >m>0$ in $\Omega$. If $N\left(A^{\g I}_{\g u}\right)\neq \{0\}$, then $A^{\g I}_{\g u}: L^2(\Omega) \to H^{-1}(\Omega,\R^d)$ has closed range.

\end{theorem}

\begin{proof}
According to Theorem \ref{theo:sobolev}, there exists $\mu_0 \in W^{1,p}(\Omega)$ such that $N\left(A^{\g I}_{\g u}\right) = \Span\ \{\mu_0\}$. By construction of $\mu_0$, there exists a constant $\tilde{m}\geq 0$ such that $\mu_0>\tilde{m}$ in $\Omega$. Take $\g f \in R(A^{\g I}_{\g u})$ and $\mu\in \left\{\mu_0\right\}^\perp$ such that $A^{\g I}_{\g u}(\mu)=\g f$. Define $\alpha=\frac{\mu}{\mu_0}\in L^2(\Omega)$. One can write
\begin{align*}
-\nabla \cdot (\alpha \mu_0 S) = &\g f \quad \tin H^{-1}(\Omega,\R^d), \\
-\mu_0 S\nabla \alpha = &\g f \quad \tin H^{-1}(\Omega,\R^d),
\end{align*}
which makes sense because $\mu_0 S \in W^{1,p}(\Omega)$ (see Lemma \ref{lem:w1p}) and $\nabla\cdot (\mu_0 S)=0$. Multiplying by $\mu_0^{-1}S^{-1} \in W^{1,p}$ yields
\begin{align*}
 \nabla \alpha =\frac{S^{-1}}{\mu_0} \g f \quad \tin H^{-1}(\Omega,\R^d) .
\end{align*}
Using Lemma \ref{lem:w1p}, it follows that
\begin{align*}
\left\Vert\nabla \alpha\right\Vert_{H^{-1}(\Omega)} \leq\norm{\frac{S^{-1}}{\mu_0}}{W^{1,p}(\Omega)} \Vert\g  f\Vert_{H^{-1}(\Omega)}.
 \end{align*}
Since $\mu=\alpha\mu_0 \in \{\mu_0\}^\perp$, we have $\int_\Omega \alpha \mu_0^2 = 0$ and by Lemma \ref{lem_cr}, there exists $c>0$ such that 

\begin{align*}
\norm{\alpha}{L^2(\Omega)} &\leq c \norm{\nabla \alpha}{H^{-1}(\Omega)} \\
&\leq c \norm{\frac{S^{-1}}{\mu_0}}{W^{1,p}(\Omega)}\norm{\g f}{H^{-1}(\Omega)}.
\end{align*}
Turning now  to $\mu=\alpha \mu_0$, we obtain that
\begin{align*}
\norm{\mu}{L^2(\Omega)} &\leq \norm{\alpha}{L^2(\Omega)} \norm{\mu_0}{L^\infty(\Omega)} \\ &\leq c\norm{\mu_0}{L^\infty(\Omega)}\norm{\frac{S^{-1}}{\mu_0}}{W^{1,p}(\Omega)} \norm{\g f}{H^{-1}(\Omega)},\\
&\leq c\norm{\mu_0}{L^\infty(\Omega)}\norm{\frac{S^{-1}}{\mu_0}}{W^{1,p}(\Omega)} \norm{\nabla \cdot (\mu S )}{H^{-1}(\Omega)}.
\end{align*}

\end{proof}

\begin{theorem}[Closed range with strain in $W_\pw^{1,p}$]\label{theo:closedrangepw}
Take $\g u$ such that $S:=\nabla^s \g u \in W_\pw^{1,p}(\Omega,\R^{d\times d})$, and that $| \text{det}\ \nabla^s \g u| \geq m>0$ in $\Omega$. If $ N\left(A^{\g I}_{\g u}\right)\neq \{0\}$, then $A^{\g I}_{\g u}: L^2(\Omega) \to H^{-1}(\Omega,\R^d)$ has closed range.

\end{theorem}

\begin{proof} According to Theorem \ref{theo:sbv}, there exists $\mu_0 \in L^2(\Omega)$ and $\norm{\mu_0}{L^2(\Omega)}=1$ and such that $N(A^{\g I}_{\g u}) = \text{span}\ \{\mu_0\}$. By construction, there exists $m>0$ such that $|\mu_0|\geq \tilde m$. Take $\g f\in R(A^{\g I}_{\g u})$ and $\mu\in L^2(\Omega)$ such that $A^{\g I}_{\g u}\mu=\g f$. Define $\alpha=\frac{\mu}{\mu_0}\in L^2(\Omega)$. As $S\in W_\pw^{1,p}(\Omega,\R^{d\times d})$, there exists a domain decomposition $\Omega_1,\dots,\Omega_k$. Note that, for $\alpha_i$ and $\g f_i$ the restrictions to $\Omega_i$, we have for any $i$,

\begin{equation}\nonumber
A^{\g I}_{\g u}(\alpha_i\mu_0) = \g f_i \quad\text{ in } H^{-1}(\Omega_i).
\end{equation}
Following the proof of the previous theorem, we can control $\nabla\alpha_i$ in $H^{-1}$ norm. There exists $c_i>0$ such that 

\begin{equation}\nonumber
\norm{\nabla \alpha_i}{H^{-1}(\Omega_i)}\leq c_i\norm{\g f}{H^{-1}(\Omega_i)}.
\end{equation}
Denote $\alpha_{\Omega_i}:=|\Omega_i|^{-1}\int_{\Omega_i}\alpha_i$. From Lemma \ref{lem_cr}, it follows that

\begin{equation}\nonumber
\norm{\alpha_i-\alpha_{\Omega_i}}{L^2(\Omega_i)}\leq c_i\norm{\g f}{H^{-1}(\Omega_i)}.
\end{equation}
Then, taking $C_1=\max\{c_i\}$ gives

\begin{equation}\label{eq:1}
\norm{\alpha_i-\alpha_{\Omega_i}}{L^2(\Omega_i)}\leq C_1\norm{\g f}{H^{-1}(\Omega)}\quad\forall i=1\dots k.
\end{equation}
Consider now the following decomposition:

\begin{equation}\label{eq:1.1}\begin{aligned}
\norm{\alpha}{L^2(\Omega)}^2=\sum_{i=1}^k\norm{\alpha_i}{L^2(\Omega_i)}^2\leq \sum_{i=1}^k\norm{\alpha_i-\alpha_{\Omega_i}}{L^2(\Omega_i)}^2+|\Omega_i|\alpha_{\Omega_i}^2.
\end{aligned}\end{equation}
It remains to prove that the mean values $\alpha_{\Omega_i}$ are controlled by $\g f$. Going back to the variational formulation,

\begin{equation}\nonumber
\int_\Omega\alpha\mu_0S:\nabla\g v=\left<\g f,\g v\right>_{H^{-1},H^1_0},
\end{equation}
we decompose it as follows:

\begin{equation}\nonumber\begin{aligned}
\sum_{i=1}^k\int_{\Omega_i}\alpha_i\mu_0S:\nabla\g v &=\left<\g f,\g v\right>_{H^{-1},H^1_0},\\
\sum_{i=1}^k\int_{\Omega_i}(\alpha_i-\alpha_{\Omega_i})\mu_0S:\nabla\g v + \alpha_{\Omega_i}\int_{\Omega_i}\mu_0S:\nabla\g v &=\left<\g f,\g v\right>_{H^{-1},H^1_0}.
\end{aligned}\end{equation}
Recalling that $\mu_0S$ is divergence free, we write $\int_{\Omega_i}\mu_0S:\nabla\g v=-\int_{\partial\Omega_i}\mu_0S\g n_i\cdot \g v$ to get that

\begin{equation}\nonumber
\sum_{i=1}^k\alpha_{\Omega_i}\int_{\partial\Omega_i}\mu_0S\g n_i\cdot \g v = \sum_{i=1}^k\int_{\Omega_i}(\alpha_i-\alpha_{\Omega_i})\mu_0S:\nabla\g v - \left<\g f,\g v\right>_{H^{-1},H^1_0}.
\end{equation}
Hence, we obtain that

\begin{equation}\nonumber
\sum_{i=1}^k\alpha_{\Omega_i}\int_{\partial\Omega_i}\mu_0S\g n_i\cdot \g v \leq \left(1+\norm{\mu_0S}{L^\infty(\Omega)}\sum_{i=1}^kc_i\right)\norm{\g f}{H^{-1}(\Omega)}\norm{\g v}{H^1_0(\Omega)}.
\end{equation}

Consider now a boundary $\Gamma_{ij}:=\partial\Omega_i\cap\partial\Omega_i\neq \emptyset$. Remark that the left and right normal traces of $\mu_0S$ are the same (divergence free jump condition). That is to say that $\mu_0S\g n_i=-\mu_0S\g n_j$ on $\Gamma_{ij}$ and $\mu_0S\g n_i$ belongs to $H^{\frac 12}(\Gamma_{ij})$. Consider the continuous extension operator $R_{ij}:H^{\frac 12}(\Gamma_{ij})\rightarrow H^1_0(\Omega)$ defined by $R_{ij}u|_{\Gamma_{ij}}=u$ and such that $R_{ij}u$ vanishes on all other boundaries $\Gamma_{pq}$ where $(p,q)\neq(i,j)$. These operators exist because the boundaries are distant from one another and their continuity constants can be chosen without being dependent on $\Gamma_{ij}$.

Taking now the test function $\g v_{ij}=R_{ij}(\alpha_{\Omega_i}-\alpha_{\Omega_j})\mu_0S\g n_i\in H^1_0(\Omega,\R^d)$ and using it in the last equation gives

\begin{equation}\nonumber
|\alpha_{\Omega_i}-\alpha_{\Omega_j}|^2\int_{\Gamma_{ij}}|\mu_0S\g n_i|^2 \leq c_R\left(1+\norm{\mu_0S}{L^\infty(\Omega)}\sum_{i=1}^kc_i\right)\norm{\g f}{H^{-1}(\Omega)}|\alpha_{\Omega_i}-\alpha_{\Omega_j}|\norm{\mu_0S\g n_i}{H^{\frac{1}{2}}(\Gamma_{ij})},
\end{equation}
where $c_R$ is such that $\norm{R_{ij}u}{H^1_0(\Omega)}\leq c_R\norm{u}{H^{\frac{1}{2}}(\Gamma_{ij})}$ for all $\Gamma_{ij}$ and $u\in H^{\frac{1}{2}}(\Gamma_{ij})$. Note that the constant $\int_{\Gamma_{ij}}|\mu_0S\g n_i|^2$ cannot be zero because $|\mu_0|\geq \tilde m>0$ and $|\det S|\geq m>0$. To summarise, we have shown that there exists a constant $C_2>0$ depending only on $\mu_0$, $S$,  and the decomposition $(\Omega_i)_{i=1}^k$ such that

\begin{equation}\label{eq:2}
|\alpha_{\Omega_i}-\alpha_{\Omega_j}| \leq C_2\norm{\g f}{H^{-1}(\Omega)},
\end{equation}
for all $i,j$ such that $\Omega_i$ and $\Omega_j$ share a boundary. This clearly can be extended to non-adjacent subdomains by transitivity and triangular inequality.

We now use the fact that $\mu\in\{\mu_0\}^\perp$, that is, $\int_\Omega\alpha\mu_0^2=0$. In other terms, 

\begin{equation}\nonumber
\sum_{i=1}^k\int_{\Omega_i}\alpha_i\mu_0^2=0,
\end{equation}
or

\begin{equation}\nonumber
\sum_{i=1}^k\alpha_{\Omega_i}\int_{\Omega_i}\mu_0^2= -\sum_{i=1}^k\int_{\Omega_i}(\alpha_i-\alpha_{\Omega_i})\mu_0^2.
\end{equation}
We deduce from \eqref{eq:1} that there exists a constant $C_3>0$ such that

\begin{equation}\label{eq:3}
\left|\sum_{i=1}^k\alpha_{\Omega_i}\int_{\Omega_i}\mu_0^2\right|\leq C_3\norm{\g f}{H^{-1}(\Omega)}.
\end{equation}
From $\eqref{eq:2}$ and $\eqref{eq:3}$,  we can now bound all the $\alpha_{\Omega_i}$ by just writing

\begin{equation}\nonumber\begin{aligned}
\left|\sum_{i=1}^k\alpha_{\Omega_i}\int_{\Omega_i}\mu_0^2\right| &= \left|\alpha_{\Omega_j}\int_{\Omega}\mu_0^2 + \sum_{i\neq j}(\alpha_{\Omega_i}-\alpha_{\Omega_j})\int_{\Omega_i}\mu_0^2\right| \leq C_3\norm{\g f}{H^{-1}(\Omega)},\\
|\alpha_{\Omega_j}|\int_{\Omega}\mu_0^2 &\leq C_3\norm{\g f}{H^{-1}(\Omega)} + \sum_{i\neq j}|\alpha_{\Omega_i}-\alpha_{\Omega_j}|\int_{\Omega_i}\mu_0^2,
\end{aligned}\end{equation}
and finally obtaining

\begin{equation}\nonumber
|\alpha_{\Omega_j}|\leq (C_3+kC_2)\norm{\g f}{H^{-1}(\Omega)},\quad\forall j.
\end{equation}
Combining \eqref{eq:1}, \eqref{eq:1.1}, and the previous inequality, we arrive at

\begin{equation}\nonumber
\norm{\alpha}{L^2(\Omega)}\leq C\norm{\g f}{H^{-1}(\Omega)}.
\end{equation}
Turning to $\mu=\alpha \mu_0$, we have
\begin{align*}
\norm{\mu}{L^2(\Omega)} &\leq \norm{\alpha}{L^2(\Omega)} \norm{\mu_0}{L^\infty(\Omega)} \\ &\leq C\norm{\mu_0}{L^\infty(\Omega)} \norm{\g f}{H^{-1}(\Omega)} \\
&\leq C\norm{\mu_0}{L^\infty(\Omega)} \norm{A^{\g I}_{\g u}\mu}{H^{-1}(\Omega)}.
\end{align*}
Hence, the proof is complete.

\end{proof}

\subsection{Stability estimates in $L^2(\Omega)$}
\label{sec:stability}

\begin{theorem}[Stability estimate for the null space estimation] \label{them:stability}
Consider a displacement field $\g u$ such that $\nabla^s\g u\in L^\infty(\Omega,\R^{d\times d}_\sym)$ and $A^{\g I}_{\g u}$ has closed range. Take $\mu\in L^2(\Omega)$ such that $\norm{\mu}{L^2(\Omega)}=1$ and $A^{\g I}_{\g u}(\mu)=0$. Take $\tilde{\g u}\in H^1(\Omega,\R^d)$ such that $\nabla^s\tilde{\g u}\in L^\infty(\Omega,\R^{d\times d}_\sym)$ and consider 
\begin{align*}
\tilde{\mu}:=\underset{\norm{\mu'}{L^2(\Omega)}=1,\  \int_\Omega{\mu' \mu}>0}{\mathrm{argmin}} \norm{A^{\g I}_{ \tilde{\g u}}(\mu')}{H^{-1}(\Omega)}.
\end{align*}
Then 
\begin{align*}
\norm{\tilde{\mu}-\mu}{L^2(\Omega)} \leq C \norm{\nabla^s \tilde{\g u} - \nabla^s\g u }{L^\infty(\Omega)}
\end{align*}
for some constant $C$ independent of $\tilde{\g u}$ and $\tilde{\mu}$. 
\end{theorem}
\begin{proof}
Write $\tilde{\mu}=\alpha \mu + \nu $, with $\nu \perp \mu$, $\alpha\in [0,1]$. Pythagoras theorem gives \mbox{$\alpha^2 + \norm{\nu}{L^2(\Omega)}^2= 1$} and $\norm{\tilde{\mu} - \mu}{L^2(\Omega)}^2 = (\alpha-1)^2 + \norm{\nu}{L^2(\Omega)}^2 = (\alpha-1)^2 + (1- \alpha^2) = 2(1-\alpha)$.
Since $1-\alpha\leq 1- \alpha^2 = \norm{\nu}{L^2(\Omega)}^2$, $$\norm{\tilde{ \mu} -\mu}{L^2(\Omega)}^2 \leq 2\norm{\nu}{L^2(\Omega)}^2.$$
Since $A^{\g I}_{\g u}$ has the closed range property, Theorem \ref{theo:closedrangew1p} yields
\begin{align*}
\norm{\nu}{L^2(\Omega)} &\leq c \norm{A^{\g I}_{\g u}(\nu)}{H^{-1}(\Omega)} \\ 
&\leq c \norm{A^{\g I}_{\g u}(\tilde{\mu}-\alpha \mu)}{H^{-1}(\Omega)} \\
&\leq c \norm{A^{\g I}_{\g u}(\tilde{\mu})}{H^{-1}(\Omega)}  \\
&\leq c \left( \norm{A^{\g I}_{\tilde{\g u}}(\tilde{\mu})}{H^{-1}(\Omega)}  + \norm{\left[A^{\g I}_{\g u} - A^{\g I}_{\tilde{\g u}}\right](\tilde{\mu})}{H^{-1}(\Omega)} \right) \\
&\leq c \left(\norm{A^{\g I}_{\tilde{\g u}}(\mu)}{H^{-1}(\Omega)} + \norm{\left[A^{\g I}_{\g u} - A^{\g I}_{\tilde{\g u}}\right](\tilde{\mu})}{H^{-1}(\Omega)}\right) \\
&\leq c \left( \norm{\left[A^{\g I}_{\tilde{\g u}} - A^{\g I}_{\g u}\right](\mu)}{H^{-1}(\Omega)} + \norm{\left[A^{\g I}_{\g u} - A^{\g I}_{\tilde{\g u}}\right](\tilde{\mu})}{H^{-1}(\Omega)}\right).
\end{align*}
Since \begin{align*}
\norm{\left[A^{\g I}_{\tilde{\g u}} - A^{\g I}_{\g u}\right](\mu)}{H^{-1}(\Omega)} = &\norm{\nabla \cdot\left[ \left( \nabla^s\tilde{\g u}-\nabla^s \g u\right)\mu\right]}{H^{-1}(\Omega)} \\ 
&\leq \norm{\left( \nabla^s\tilde{\g u}-\nabla^s \g u\right)\mu}{L^2(\Omega)} \\
&\leq \norm{ \nabla^s\tilde{\g u}-\nabla^s \g u}{L^\infty(\Omega)} \norm{\mu }{L^2(\Omega)} \\
&\leq \norm{ \nabla^s\tilde{\g u}-\nabla^s \g u}{L^\infty(\Omega)} 
\end{align*}
and \begin{align*}
\norm{\left[A^{\g I}_{\g u} - A^{\g I}_{\tilde{\g u}}\right](\tilde{\mu})}{H^{-1}(\Omega)} \leq \norm{ \nabla^s\tilde{\g u}-\nabla^s \g u}{L^\infty(\Omega)},
\end{align*} the following holds:
\begin{align*}
\norm{\tilde{\mu} - \mu }{L^2(\Omega)} \leq 2\sqrt{2}c \norm{ \nabla^s\tilde{\g u}-\nabla^s \g u}{L^\infty(\Omega)}.
 \end{align*}

\end{proof}

\begin{theorem}[General stability estimate] Consider two displacement fields $\g u, \tilde{\g u}$ such that $\nabla^s\g u$ and such that $ \nabla^s\tilde{\g u}\in L^\infty(\Omega,\R^{d\times d}_\sym)$ and $A^{\g I}_{\g u}$ has closed range. Take a real number $r>0$ and $\mu,\tilde{\mu}\in L^2(\Omega)$ respectively solutions of

\begin{equation}\nonumber
A_{\g u}^{\g I}(\mu)=\g f,\quad \int_\Omega\mu=1,\quad \norm{\mu}{L^2(\Omega)}\leq r,
\end{equation}
\begin{equation}\nonumber
A_{\tilde{\g u}}^{\g I}(\tilde \mu)=\tilde{\g f},\quad \int_\Omega\tilde\mu=1,\quad \norm{\tilde\mu}{L^2(\Omega)}\leq r.
\end{equation}
There exists a constant $C>0$ independent on $\tilde{\g u}, \tilde\mu, \tilde{\g f}$ such that,

 \begin{equation}\nonumber
 \norm{\tilde\mu-\mu}{L^2(\Omega)}\leq C\left(\norm{\tilde{\g f}-\g f}{H^{-1}(\Omega)^d}+r\norm{\nabla^s\tilde{\g u}-\nabla^s\g u}{L^\infty(\Omega)}\right).
 \end{equation}
\end{theorem}

\begin{proof}  By difference, we write that $A_{\g u}^{\g I}(\tilde\mu-\mu) = \tilde{\g f}-\g f + (A_{\g u}^{\g I}-A_{\tilde{\g u}}^{\g I})\tilde\mu$. If $N(A_{\g u}^{\g I})=\{0\}$ then $\tilde\mu-\mu\in N(A_{\g u}^{\g I})^\perp$ and applying the closed range property in the same manner than for the previous Theorem we get
 
 \begin{equation}\nonumber
 \norm{\tilde\mu-\mu}{L^2(\Omega)}\leq C\left(\norm{\tilde{\g f}-\g f}{H^{-1}(\Omega)}+r\norm{\nabla^s\tilde{\g u} - \nabla^s{\g u}}{L^\infty(\Omega)}\right).
 \end{equation}
 If $N(A_{\g u}^{\g I})=\Span\{\mu_0\}$, where $\norm{\mu_0}{L^2(\Omega)}=1$ and $m_0:=\int_\Omega\mu_0\neq 0$, we decompose $\mu$ as $\mu=\alpha\mu_0+\nu$ and $\tilde{\mu}$ as $\tilde{\mu}=\tilde\alpha\mu_0+\tilde\nu$ where $\nu,\tilde\nu\in N(A_{\g u}^{\g I})^\perp$. Then we have
 
  \begin{equation}\label{eq:nu}
 \norm{\nu - \tilde\nu}{L^2(\Omega)}\leq C\left(\norm{\tilde{\g f}-\g f}{H^{-1}(\Omega)}+r\norm{\nabla^s\tilde{\g u} - \nabla^s{\g u}}{L^\infty(\Omega)}\right).
  \end{equation}
  By Pythagoras' theorem,  $\norm{\tilde\mu-\mu}{L^2(\Omega)}^2 = (\tilde\alpha-\alpha)^2 + \norm{\tilde\nu - \nu}{L^2(\Omega)}^2$.
 Now using that $$\int_\Omega(\tilde\mu-\mu)=(\tilde\alpha-\alpha)\int_\Omega \mu_0 + \int_\Omega (\tilde\nu - \nu)=0,$$ we get that
  
  \begin{equation}\nonumber\begin{aligned}
    (\tilde\alpha-\alpha)^2m_0^2 &\leq |\Omega|\norm{\tilde\nu - \nu}{L^2(\Omega)}^2,\\
  \end{aligned}\end{equation}
  and hence, 
  
    \begin{equation}\label{eq:alpha}\begin{aligned}
    \norm{\tilde\mu-\mu}{L^2(\Omega)}^2\leq\left(1+\frac{|\Omega|}{m_0^2}\right)\norm{\tilde\nu - \nu}{L^2(\Omega)}^2.
  \end{aligned}\end{equation}  
 We conclude by combining inequalities \eqref{eq:nu} and \eqref{eq:alpha}.
\end{proof}

\section{Numerical experiments in the static case}\label{sec:num}

The objective here is to numerically reconstruct an elasticity tensor $\g C(x)$  in a smooth domain $\Omega\subset\widetilde \Omega\in\R^2$
from the knowledge of a set of data $\{ (\g u^{\ell}, \g f^{\ell}) \}_{\ell = 1}^n$  satisfying the linear elasticity equation 

\begin{equation}\nonumber
 - \nabla \cdot \left( \g C : \nabla^s \g u^{\ell} \right) = \g 0.
\end{equation}

\subsection{Forward problem and data generation}

In order to generate different displacement fields $\g u^\ell\in H^1(\Omega,\R^2)$ of static elastic deformation, we use the classic finite elements approach to solve the boundary-value problem

\begin{equation}\nonumber\left\{\begin{aligned}
 - \nabla \cdot \left( \g C : \nabla^s \g u^{\ell} \right) &= \g 0 \quad\text{ in }\widetilde\Omega,\\
 \g u^\ell &=\g 0 \quad\text{ on } \Gamma_{\text{Dir}},\\
 \left(\g C : \nabla^s \g u^{\ell} \right)\cdot\g n &=\g g^\ell \quad\text{ on } \Gamma_{\text{Neu}},\\
  \left(\g C : \nabla^s \g u^{\ell} \right)\cdot\g n &=\g 0 \quad\text{ elsewhere on }\partial\widetilde\Omega, \\
\end{aligned}\right.\end{equation}
where $\g g^\ell$ could be any surface force density. In the simulations, we use $\widetilde \Omega = (-1,1)^2$ and $\Gamma_{\text{Dir}}$ and $\Gamma_{\text{Neu}}$ are as described in Figure \ref{fig:forward}.

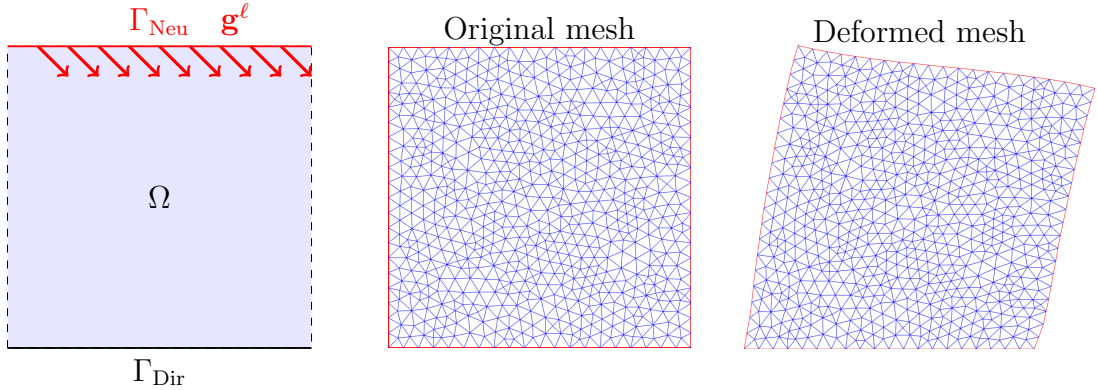
\begin{figure}
\centering
\begin{tikzpicture}[scale=2]
\draw[dashed, line width = 0, fill = blue!10] (-1,-1) -- (-1,1) --(1,1) --(1,-1) -- (-1,-1);
\draw[line width = 0.7,red] (1,1) -- node[above]{$\Gamma_{\text{Neu}}$}(-1,1) ;
\draw[line width = 0.7] (-1,-1) -- (1,-1) ;

\draw (0,-1) node[below]{$\Gamma_{\text{Dir}}$};
\draw (0,0) node{$\Omega$};
\draw[red] (0.5,1) node[above]{$\g g^\ell$};

\foreach \x in {1,...,9}
\draw[line width = 1.2,->,red] (-1+0.2*\x,1) -- +(0.2,-0.2) ;
\node at (2.5,0){\includegraphics[width=.25\textwidth]{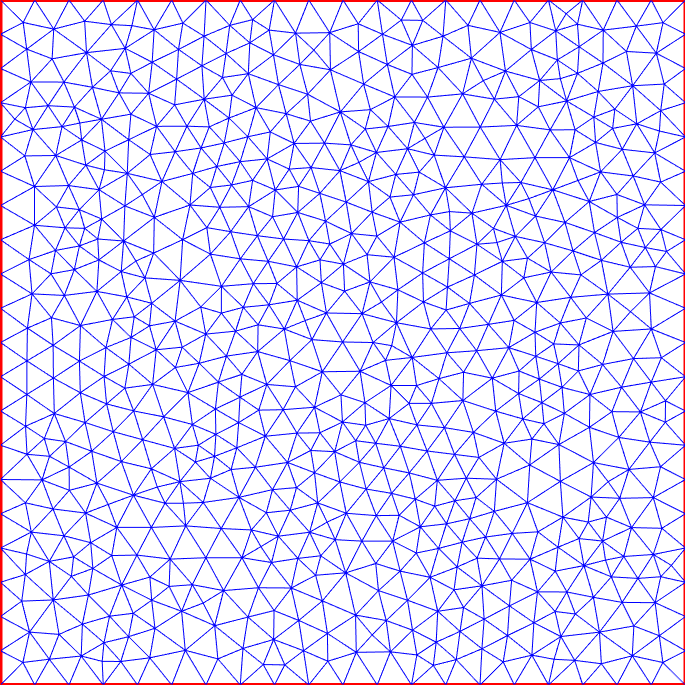}};
\node at (5,0){\includegraphics[width=0.29\textwidth]{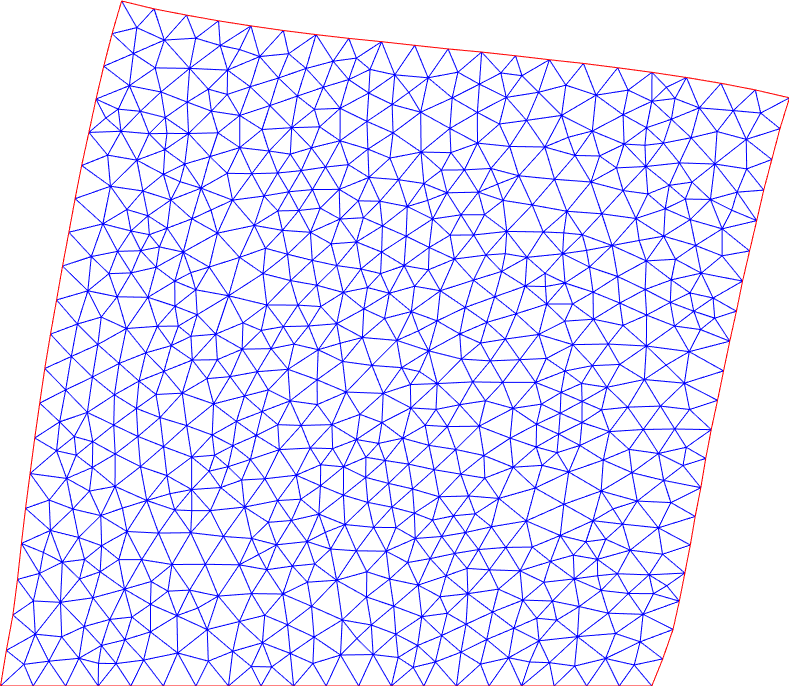}};

\node at (2.5,1.1){Original mesh};
\node at (5,1.1){Deformed mesh};

\end{tikzpicture}

\label{fig:forward}
\caption{Numerical experimental setting: a non-structured meshing of the domain $\Omega$ and the computed  elastic deformation.}
\end{figure}

The computations of direct data are made using the $\cP^0-\cP^1$ finite elements method. The solution $\g u^\ell$ is then interpolated and recorded on a structured Cartesian grid over $\widetilde \Omega$. The inverse problem is solved on a subdomain $\Omega\subset\widetilde \Omega$ endowed with a new non-structured mesh. This prevents from committing an inverse crime in inverting the problem using the same discrete operator as the one used for the direct problem.

\subsection{Finite elements discretization}

We assume here without loss of generality  that the chosen domain $\Omega\in\R^2$ is polygonal that admits an exact triangular mesh $\T_h = \{ T_{i} \}_{i=1}^{N_\T}$ for any small maximum edge length $h>0$. More precisely, $\T_h$ is a set of $N_T$ open triangles such that $T_i\cap T_j=\emptyset$ if $i\neq j$ and $\overline\Omega=\bigcup_{i=1}^{N_T}\overline{T_i}$.  Let us introduce the classic function spaces associated to $\T_h$:

\begin{itemize}
 \item[(i)] The space $\cP^0(\T_h)$  is the set of functions that are constant on each triangle:
 
 \begin{equation}\nonumber
  \cP^0(\T_h) = \left\{ u \in L^{2}(\Omega,\R),\   \forall i,\  u|_{T_i} \text{ is constant} \right\}.
 \end{equation}
Its canonical basis  $(\xi_j)_{j=1}^{N_T}$ is given by $\xi_j:=\1_{T_i}$.
 \item[(ii)]  The space $\cP^1(\T_h)$ is
 the set of continuous functions that are linear on each triangle $T_i$:  
$$
  \cP^1(\T_h) = \left\{ v \in H^1( \Omega),\   \forall i,\ v|_{T_i} \text{ is linear} \right\}.
$$ 
Its canonical basis  $(e_i)_{i=1}^{N_n}$ is defined by $e_i(x_j)=\delta_{ij}$ for any node $x_j$ of the triangulation $\T_h$.
 \item[(iii)]  The space $\cP^1({\T_h},\R^2)$ is the set of vector-valued $\cP^1(\T_h)$ functions. Its canonical basis is denoted $(\g e_i)_{i=1}^{2N_n}$.
 
  \item[(iv)]  The space $\cP^1_0({\T_h})$, (resp. $\cP^1_0({\T_h},\R^2)$) is the space of $\cP^1(\T_h)$ (resp. $\cP^1(\T_h,\R^2)$) functions that vanish on $\partial\Omega$:
  
  \begin{equation}\nonumber
  \cP^1_0({\T_h}):=\cP^1({\T_h})\cap H^1_0(\Omega)\quad\text{ and }\quad  \cP^1_0({\T_h},\R^2):=\cP^1({\T_h})\cap H^1_0(\Omega,\R^2).
  \end{equation}
  
\item[(v)]  Its canonical basis is denoted by $(\tilde{e}_i)_{i=1}^{N_\text{int}}\subset (e_i)_{i=1}^{N_n}$ (resp. $(\tilde{\g e}_i)_{i=1}^{2N_\text{int}}\subset (\g e_i)_{i=1}^{2N_n}$) where $N_\text{int}$ is the number of internal nodes of the mesh.

\end{itemize}

Scalar functions $\mu^{(k)}$ and  displacement fields $\g u^{\ell}$ are then projected respectively on the bases of $\cP^0(\T_h)$ and $\cP^1(\T_h)$:
$$ \mu^{(k)}(x) = \sum_{j=1}^{N_T} \mu^{(k)}_j \xi_j(x) \quad \text{ and }\quad \g u^{\ell}(x) = \sum_{i=1}^{2 N_n} u^{\ell}_i \g e_{i}(x).$$

\subsection{Discrete formulation of the inverse problem}

We assume the knowledge of a model for the elasticity tensors $\g C$ of the form
  
$$\g C(x) = \sum_{k=1}^{N} \mu^{(k)}(x) \g C^{k},$$
where  all the unknown scalar fields $\mu^{(k)}$  belong to $L^{2}(\Omega)$ and the constant tensors $\g C^k$ are known.
Recall that the reconstruction problem of each fields $\mu^{(k)}$ reads as the linear problem

\begin{align}
 \begin{pmatrix}
  A^{\g C^{1}}_{{\g u}^{1}} & \ldots &  A^{\g C^{N}}_{{\g u}^{1}} \\
  \vdots &  & \vdots \\
   A^{\g C^{1}}_{{\g u}^{n}} & \ldots &  A^{\g C^{N}}_{{\g u}^{n}} 
 \end{pmatrix} \begin{pmatrix} \mu^{1} \\ \vdots \\ \mu^{N}
  \end{pmatrix} = \begin{pmatrix} \g f^{1} \\ \vdots \\ \g f^{n}
  \end{pmatrix},
\end{align}
where the operator $A_{\g u}^{\g C}$ is defined by $A_{\g u}^{\g C} \mu = - \nabla\cdot\left(\mu \g C:\nabla^s\g u\right),$
or in a weak sense by 
\begin{equation}\label{equ:weakF_A}
 \langle A_{\g u}^{\g C} \mu , \g v \rangle_{H^{-1},H^1_0} = \int_{\Omega} \mu(x)(\g C : \nabla^s \g u(x) ) :\nabla^s \g v(x)\td x, \quad \forall \g v \in H^1_0(\Omega, \R^d).
\end{equation}
These operators admit a   straightforward finite elements discretization defining the matrices $\A^{\g C}_{\g u}\in \R^{N_\text{int}\times N_T}$ and $\F^\ell\in \R^{N_\text{int}}$ as

\begin{equation}\begin{aligned}
(\A^{\g C}_{\g u})_{ij} &:=\langle A_{\g u}^{\g C} \xi_j , \tilde{\g e}_i \rangle_{H^{-1},H^1_0} =  \int_{\Omega} \xi_j(x)(\g C : \nabla^s \g u(x) ) :\nabla^s \tilde{\g e}_i(x)\td x,\\
\F^{\ell}_{i} &:=\langle \g f^\ell, \tilde{\g e}_i \rangle_{H^{-1},H^1_0}.
\end{aligned}\end{equation}
Introducing now the block matrices

\begin{equation}\nonumber
 \A := 
 \begin{pmatrix}
  \A ^{\g C^{1}}_{{\g u}^{1}} & \ldots &  \A^{\g C^{N}}_{{\g u}^{1}} \\
  \vdots &  & \vdots \\
   \A^{\g C^{1}}_{{\g u}^{n}} & \ldots &  \A^{\g C^{N}}_{{\g u}^{n}} 
 \end{pmatrix}\in \R^{nN_\text{int}\times N N_T},\quad
 \F := \begin{pmatrix}
  \F^1  \\
  \vdots\\
   \F^n
 \end{pmatrix}\in \R^{nN_\text{int}},\quad 
  \M = \begin{pmatrix}
  \mu^1  \\
  \vdots\\
   \mu^N
 \end{pmatrix}\in \R^{N N_T},
\end{equation}
the general inverse problem admits a simple discrete projection on the finite elements spaces and reads as 

\begin{equation}
\A\M = \F.
\end{equation}

Note that,  in the static case, $\F=0$ leads to an eigenvector problem. In this case,  the formulation becomes 

\begin{equation}
\A\M = 0,\quad \norm{\M}{2}=1.
\end{equation}

As will be seen
later, in practice it might be more convenient to impose a positivity constraint over $\M:=(\mu^1,\dots,\mu^n)^T$ of the form 

\begin{equation}
\mu^{(k)} \geq \mu^{(k)}_{\min}>0,\quad \forall k,
\end{equation}
since the unknown elastic parameters are indeed positive valued functions.

\subsection{Least squares approach and regularization}

We recall that,  we  have proved for the shear modulus case that the reconstruction problem  of  $\mu$ from  $\{(\g u^{\ell}, \g f^{\ell})\}_{\ell = 1:n}$
is theoretically well-posed in the continuous setting. This is not clearly the case in the discretized version as 
the linear system $\A\M  = \mathbb{F}$ has $N N_T$  degrees of freedom with  only 
$n N_n$ equations. In practice,  we compute the fields $\mu$ by  minimizing a regularized mean squares functional of the form     
$$ J(\mu) = \| \A\M - \mathbb{F}\|^2_2 + R_{\text{TV}}(\M),$$
where the regularization  $R_{\text{TV}}(\M)$ penalizes the total variation of each $\mu^{(k)}$.  More precisely, in the case where $\mu\in\cP^0(\T_h)$ that admits the decomposition $\mu(x)=\sum_{j=1}^{N_T} \mu_j\xi_j(x)$,
we can show that  the $TV$ semi-norm can be directly expressed as a linear $L^1$-penalization. 
 As it is clear that $\cP^0(\T_h) \subset BV(\Omega)$, we have

\begin{equation}\nonumber\begin{aligned}
|\mu|_{TV(\Omega)} = \int_{\Omega}|D \mu|.
\end{aligned}\end{equation}
Call $E=\{(i,j)\in\{1,\dots,N_n\}^2,\ i\neq j,\ \partial T_i\cap \partial T_j\neq \emptyset\}$ the set of all the oriented internal edges. As $\mu$ is constant on each triangle $T_i$, the Radon measure derivative $D \mu$ is given by

\begin{equation}\nonumber
D \mu = \frac 12\sum_{(i,j)\in E}(\mu_i-\mu_j) \cH^{1}_{\partial T_i\cap \partial T_j}\nu_{ij},
\end{equation}
where $\nu_{ij}$ is the normal vector from triangles $T_i$ to $T_j$ and $\cH^{1}_{\partial T_i\cap \partial T_j}$ is the restriction of the dimension one Hausdorff measure to the edge $(i,j)$. Hence,
\begin{equation}\nonumber\begin{aligned}
|\mu|_{TV(\Omega)} = \int_{\Omega}|D \mu| = \frac 12\sum_{(i,j)\in E}|\mu_i-\mu_j| \cH^{1}(\partial T_i\cap \partial T_j).
\end{aligned}\end{equation}
We observe that if one defines the linear operator $L:\cP^0(\T_h)\rightarrow\R^{\text{card}(E)}$ by
\begin{equation}\nonumber
(L\mu)_{(i,j)\in E} = (\mu_i-\mu_j)\cH^{1}(\partial T_i\cap \partial T_j),
\end{equation}
then
\begin{equation}\nonumber\begin{aligned}
|\mu|_{TV(\Omega)} = \frac 12\norm{L\mu}{\ell^1(E)}.
\end{aligned}\end{equation}
Finally, it can be shown that the TV-regularization term  $R_{\text{TV}}(\M)$ can be expressed under the form 
\begin{equation}
 R_{\text{TV}}(\M) = \sum_{k=1}^n\e^{(k)}_{\text{TV}} \norm{ L \mu^{(k)}}{\ell^1(E)},
\end{equation}
where $\e^{(k)}_{\text{TV}}$ are regularizing parameters. The reconstruction of $\mu$ can then be computed by minimizing the functional
\begin{equation}\label{eq:optim}
J(\M) = \| \A\M - \mathbb{F}\|^2_2 +  \sum_{k=1}^n\e^{(k)}_{\text{TV}} \norm{ L \mu^{(k)}}{\ell^1(E)},
\end{equation}
subject to $\M\geq\M_{\min}$ where $\M_{\min}:=(\mu^{(1)}U,\dots,\mu^{(n)}U)^T$ with $U:=(1,\dots,1)\in \R^{N_T}$.

\begin{remark}
In practice the minimisation of \eqref{eq:optim} can be achieved with any efficient optimisation routine. Here, we used the \it CVX  Matlab toolbox \cite{cvx,gb08} which is well-adapted  to this kind of convex optimization problems under linear constraints.   
\end{remark}

\begin{remark}
As the given displacement fields $\g u^{\ell}$ are in general noisy, which significantly affect the eigenvalues of the associated operator $\mathbb{A}$, it could be convenient to introduce a beforehand smoothing of these vector fields. More precisely, a natural way is to consider an elastic regularization 
$\g u^{\ell}_{\e_{\text{elas}}}$ 
defined by
$$\g u^{\ell}_{\e_{\text{elas}}} = \text{argmin}_{\g u} \left\{ \frac{1}{\e_{\text{elas}}} \| \g u - \g u^{\ell}  \|^{2}_{L^2(\Omega)}
+\| \nabla^s \g u \|^{2}_{L^2(\Omega)}  \right\} . $$
This is also equivalent to compute $\g u^{\ell}_{\e_{\text{elas}}}$  in the finite elements context as the solution 
of the following linear system 
$$ \g u^{\ell}_{\e_{\text{elas}}}  = ( \e_{\text{elas}} M + L)^{-1}(M  \g u^{\ell}),$$
where $M$ and $L$ are respectively  the mass and the vector stiffness matrix and $\e_{\text{elas}} >0$ is a regularization parameter.   
\end{remark}

\subsection{Numerical experiments}\label{subsec:num}

The motivation  is now  to present some numerical experiments in the static case  where the tensor $\g C$ 
is assumed to be of the form:  
\begin{itemize}
 \item[(i)]  A shear modulus reconstruction only: $\g C= \mu {\g I}$;
 \item[(ii)]   A two Lam\'e parameters reconstruction: $\g C= 2 \mu {\g I} + \lambda I \otimes I$;
  \item[(iii)]  An anisotropic  stiffness reconstruction: $\g C =  \mu^{(1)} {\g C}^{1} + \mu^{(2)} {\g C}^{2} + \mu^{(3)} {\g C}^{3}$, where
  tensors ${\g C}^{1}$, ${\g C}^{2}$ and  ${\g C}^{3}$ are defined by  (\ref{add1}) and  (\ref{add2}).
\end{itemize}

We  only present  in this paper some numerical experiments in the static case 
but  other experiments have be done in the harmonic regime with similar results.
In each case, we then use the following additional constraints on $\mu^{(k)}$:   
$$ \mu^{(k)} \geq 1.$$

We first consider the simplest case of  shear medium  $\g C = \mu {\g I}$ in order to illustrate and analyze 
the influence of each of the regularization parameters $\e_{TV}$ and $\e_{\text{elas}}$ on the reconstruction. 
In particular, we will see that the reconstruction of $\mu$ is very accurate as soon as the 
choice of $\e_{TV}$ and $\e_{\text{elas}}$ is appropriate. 

We show that our methodology still works in the case of more complex tensor $\g C$. 
In particular, we highlight that the reconstruction of $(\lambda,\mu)$ in the isotropic elastic case 
and  $(\mu^{(1)},\mu^{(2)},\mu^{(3)})$ in the anisotropic shear case are also accurate provided that 
the number of the sets of data $\{\g u^{\ell}\}$ is sufficiently large. 

\subsubsection{Shear modulus inversion}

In this  subsection, we first focus on the case $\g C = \mu(x) {\g I}$, where we consider 
three different choices for the shear modulus $(\mu_1,\mu_2,\mu_3)$, which are illustrated in Figure \ref{fig:mu}. 
In each case, we compute the direct elastic vector fields $\g u^{(1)}$ associated to the same boundary conditions. 
Each solution is plotted  in Figure \ref{fig:u_mu} and we can  
observe the similarity of the different elastic fields $\g u^{(1)}$. 
Notice that the mesh used to compute the elastic vector fields has been build such as the characteristic size of each triangle
is about $h = 0.01$.

About the  reconstruction of the shear modulus $\mu$, we recall that we need  to fix only the two  
regularization parameters $\e_{TV}$ and $\e_{\text{elas}}$. In each case, we also use the same mesh where the  triangles have now a
characteristic size of the order of $h = 0.03$. 

The first experiments illustrated in Figure \ref{fig:influence_e_TV} have been done  with $\e_{\text{elas}} = 10^{-5}$. Each column corresponds to different values of $\e_{TV}$ which are respectively  equal to 
$ \e_{TV} = 10^{-6}$, $ \e_{TV} = 10^{-5}$, and $ \e_{TV} = 10^{-4}$. 
Each line corresponds to the data associated with $\mu$, $\tilde{\mu}$, and $\mu_3$. 

We can  observe that the reconstruction is perturbed if $\e_{TV}$ is too small and 
 becomes very quantitative with an appropriate choice of $\e_{TV}$. These first experiments 
 show the advantage of the TV regularization which preserves the discontinuities. 
 Finally, it shows the real possibility  of reconstructing a non-smooth shear modulus $\mu$ with 
 only one set of data $\{\g u^{(1)}\}$. 
 
The second experiments (presented in Figure \ref{fig:influence_e_elastic}) show the influence of an elastic regularization on the data. 
Indeed, the estimation  can be noisy and need in practice to be regularized.
We then try here to understand the influence of an elastic regularization on the quality of the reconstruction. 
We then fix the value of $ \e_{TV} = 10^{-4}$ and compare the reconstruction of $\mu$ obtained with 
$ \e_{\text{elas}} = 10^{-5}$, $ \e_{\text{elas}} = 10^{-4}$ and $ \e_{\text{elas}} = 10^{-3}$. It then clearly appears that the effect of 
the elastic regularization is to smooth the reconstruction of the shear modulus $\mu$. 

Finally, as  expected by our theoretical results, these experiments clearly demonstrate  the ability of our 
methodology  to  reconstruct non-smooth shear modulus  $\mu$ using only one set of data $\{\g u^{\ell}\}$.

\def\imagescaledef{0.5}
\def\textscaledef{0.5}
\def\sdef{-0.8cm}

\begin{figure}
\begin{center}
\def\imagescale{0.6}
\def\textscale{0.6}
\input{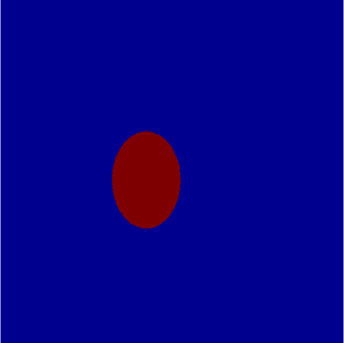}
\input{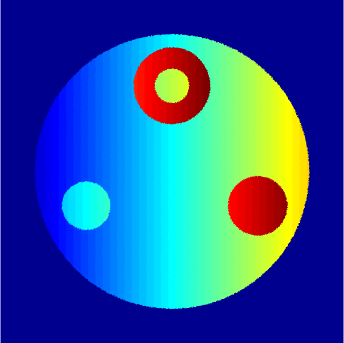}
\input{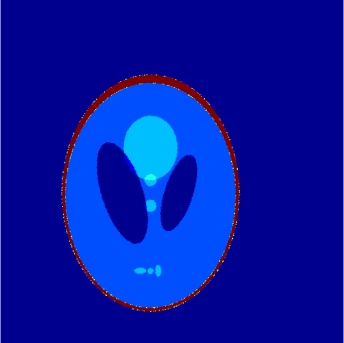}
\caption{Examples of shear modulus phantoms:  $\mu$, $\tilde{\mu}$ and $\mu_3$}
\label{fig:mu}
\end{center}
\end{figure}

\begin{figure}
\begin{center}
\def\imagescale{0.6}
\def\textscale{0.6}
\input{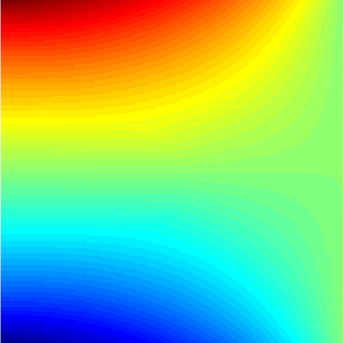}
\input{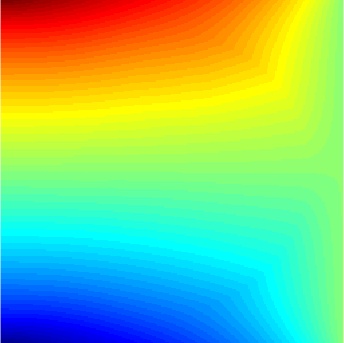}
\input{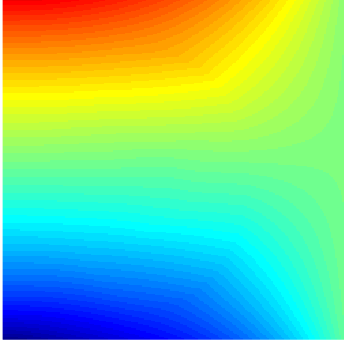} 
\input{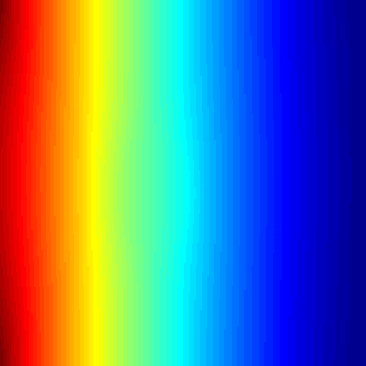}
\input{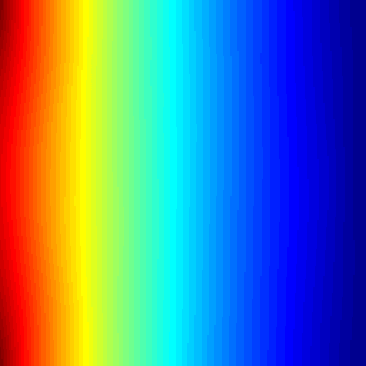}
\input{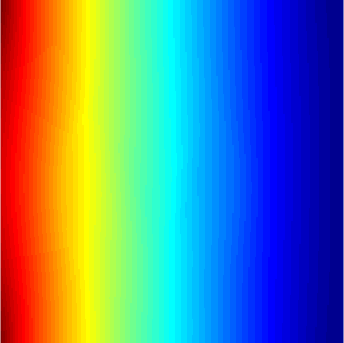} 
\caption{Lines: first and second components of vector fields $u$;  Each column (from left to right) corresponds to the use of $\mu$, $\tilde{\mu}$ and $\mu_3$, respectively.}
\label{fig:u_mu}
\end{center}
\end{figure}

\begin{figure}
\begin{center}
\def\imagescale{0.5}
\def\textscale{\textscaledef}
\def\s{\sdef}
 \input{Figures/test1_mu_mu}\hspace{\s}
 \input{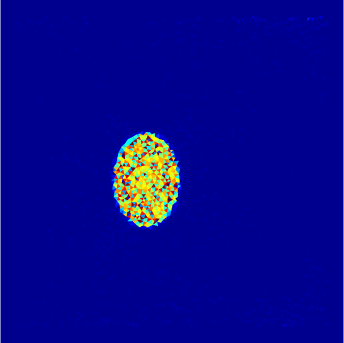}\hspace{\s}
 \input{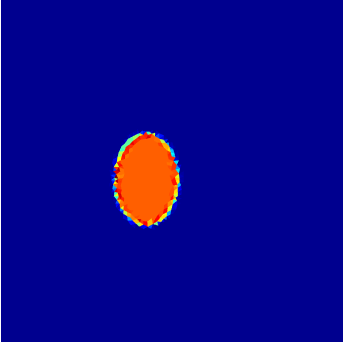}\hspace{\s}
 \input{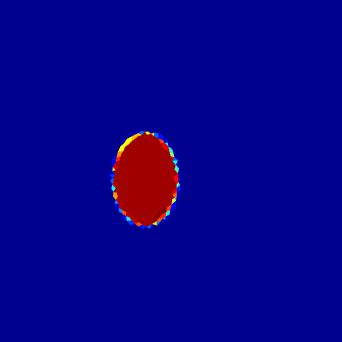} \\
 \input{Figures/test2_mu_mu}\hspace{\s}
 \input{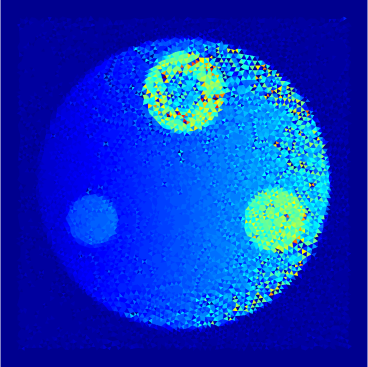}\hspace{\s}
 \input{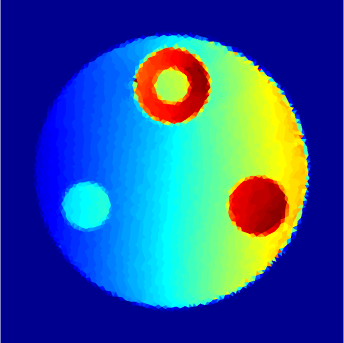}\hspace{\s}
 \input{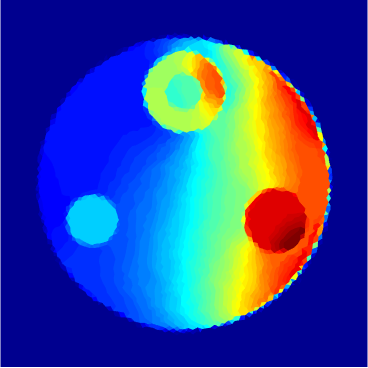} \\
 \input{Figures/test3_mu_mu}\hspace{\s}
 \input{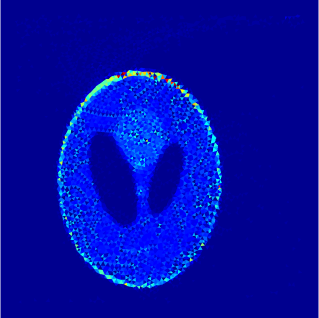}\hspace{\s}
 \input{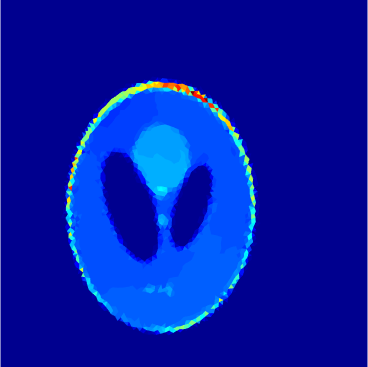}\hspace{\s}
 \input{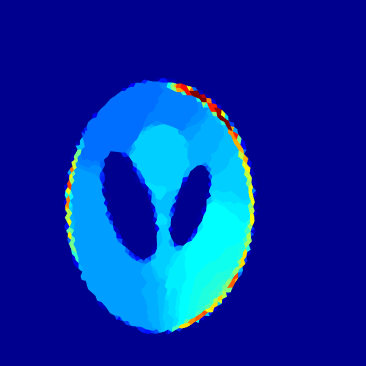} \\
 \caption{Reconstruction of $\mu$: influence of the parameter $\e_{TV}$ ; Lines: shear modulus  $\mu$, $\tilde{\mu}$  and $\mu_3$. 
 Columns: $\e_{TV} = 10^{-6}$, $\e_{TV} = 10^{-4}$ and $\e_{TV} = 10^{-3}$. In each case, we use $\e_{\text{elas}} = 10^{-5}$.}
 \label{fig:influence_e_TV}
 \end{center}
 \end{figure}

\begin{figure}
\begin{center}
\def\imagescale{0.5}
\def\textscale{\textscaledef}
\def\s{\sdef}
 \input{Figures/test1_mu_mu}\hspace{\s}
 \input{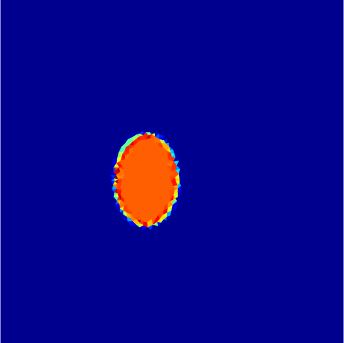}\hspace{\s}
 \input{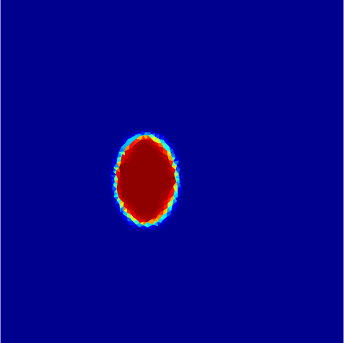}\hspace{\s}
 \input{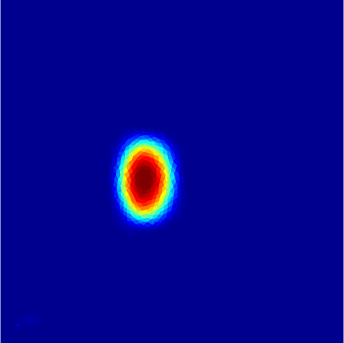} \\
 \input{Figures/test2_mu_mu}\hspace{\s}
 \input{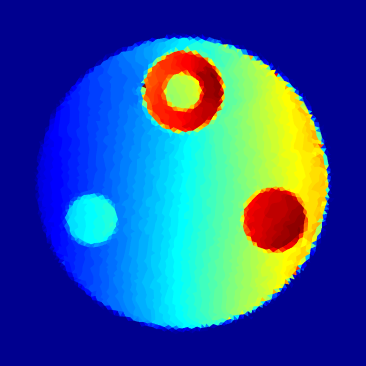}\hspace{\s}
 \input{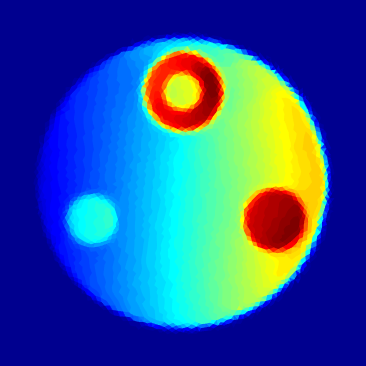}\hspace{\s}
 \input{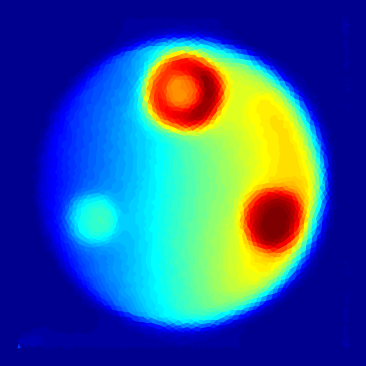} \\
 \input{Figures/test3_mu_mu}\hspace{\s}
 \input{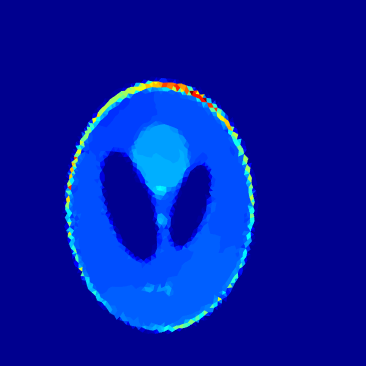}\hspace{\s}
 \input{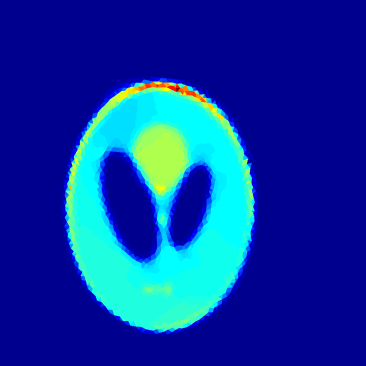}\hspace{\s}
 \input{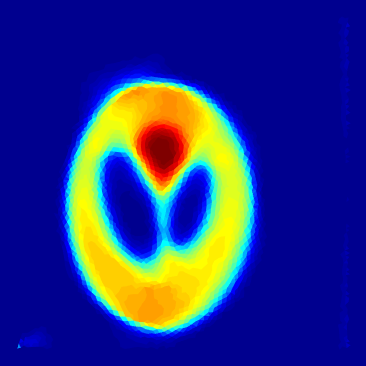} \\
 \caption{Reconstruction of $\mu$: influence of the parameter $\e_{\text{elas}}$; Lines: shear modulus  $\mu$, $\tilde{\mu}$  and $\mu_3$. 
 Columns: $\e_{\text{elas}} = 10^{-5}$, $\e_{\text{elas}} = 10^{-4}$ and $\e_{\text{elas}} = 10^{-3}$. In each case, we use $\e_{TV} = 10^{-4}$.}
 \label{fig:influence_e_elastic}
 \end{center}
 \end{figure}

\subsubsection{Two Lam\'e coefficients inversion}

We now consider the case of isotropic elasticity tensor  
$$\g C =  2 \mu {\g I} + \lambda I \otimes I.$$ 
The numerical reconstructions of Lam\'e coefficients are presented in Figures \ref{fig:mulambda_1} 
and \ref{fig:mulambda_2}, where  two different choices of Lam\'e coefficients are used.
In all experiments, we take the regularization parameters: $\e_{TV} = 10^{-4}$ and $\e_{\text{elas}} = 10^{-4}$.
Moreover, each column corresponds to the numerical reconstruction of  $(\lambda,\mu)$ obtained respectively
with $n=1$, $n=2$ and $n=4$ sets of data $\{\g u^{\ell}\}$. We also plot  the exact Lam\'e coefficient on the first column. 

Notice that in the case of one set of data, we succeeded in reconstructing a first rough approximation of $(\lambda,\mu)$. 
Finally, using  $n=2$ and $n=4$ sets of data leads to a precise reconstruction of $(\lambda,\mu)$ even for 
complex Lam\'e coefficients.

\begin{figure}
\begin{center}
\def\imagescale{0.5}
\def\textscale{\textscaledef}
\def\s{\sdef}
 \input{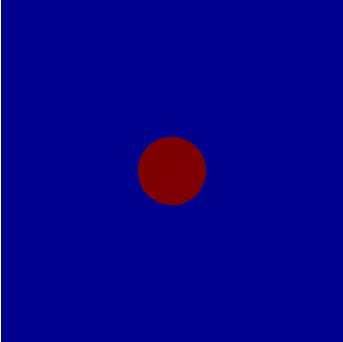}\hspace{\s}
 \input{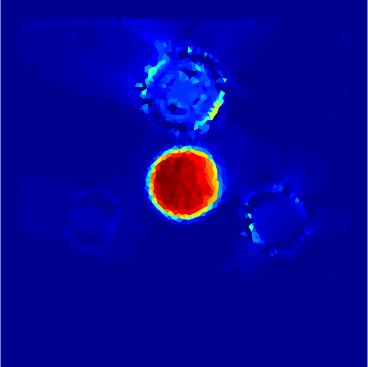}\hspace{\s}
 \input{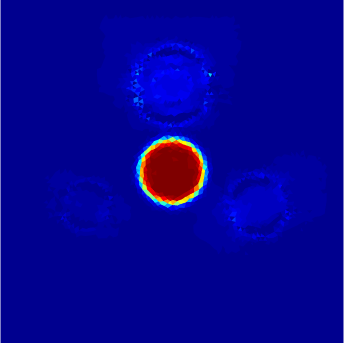}\hspace{\s}
 \input{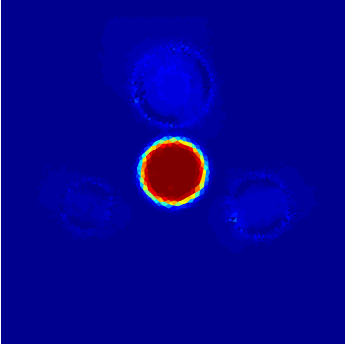} \\
 \input{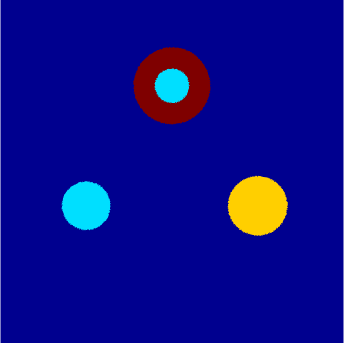}\hspace{\s}
 \input{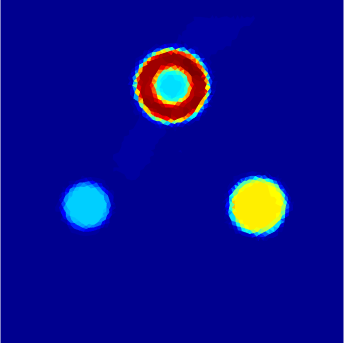}\hspace{\s}
 \input{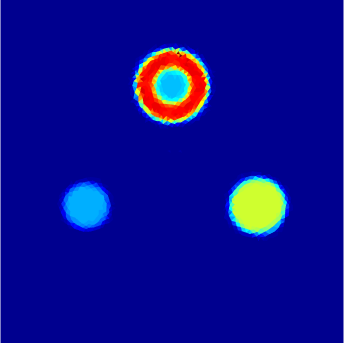}\hspace{\s}
 \input{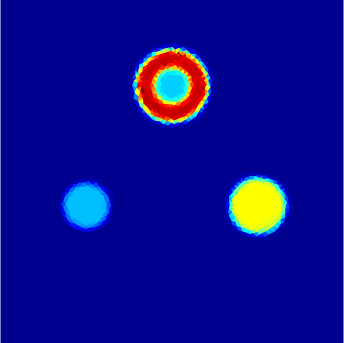} 
\caption{Reconstruction of the Lam\'e coefficients $(\lambda_1,\mu_1)$:   influence of the number of data $\{ {\g u}^{\ell}\}$ ;  
Lines: $\lambda$ and $\mu$; From left to right: Exact Lam\'e coefficients and  their reconstructions obtained respectively with
$1$, $2$ and $4$ sets of data $\{ {\g u}^{\ell}\}$. Here, we used $\e_{TV} = 10^{-4}$ and $\e_{\text{elas}} = 10^{-4}$.}
\label{fig:mulambda_1}
 \end{center}
 \end{figure}

\begin{figure}
\begin{center}
\def\imagescale{0.5}
\def\textscale{\textscaledef}
\def\s{\sdef}
 \input{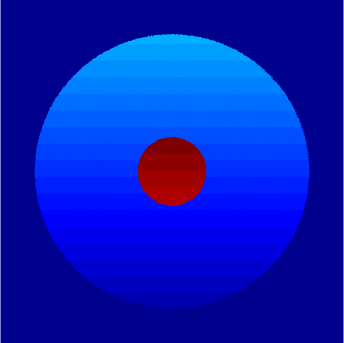}\hspace{\s}
 \input{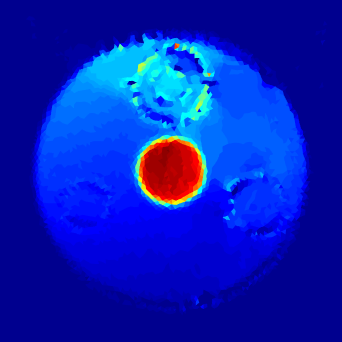}\hspace{\s}
 \input{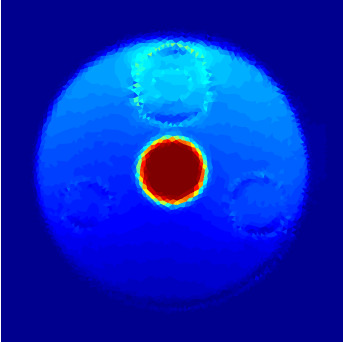}\hspace{\s}
 \input{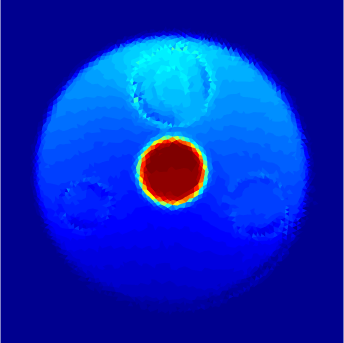} \\
 \input{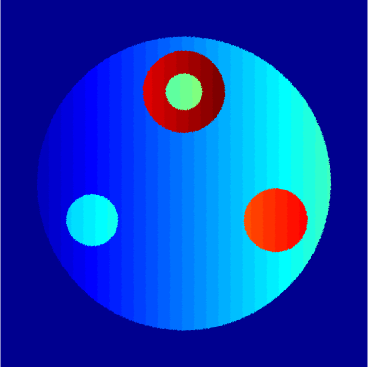}\hspace{\s}
 \input{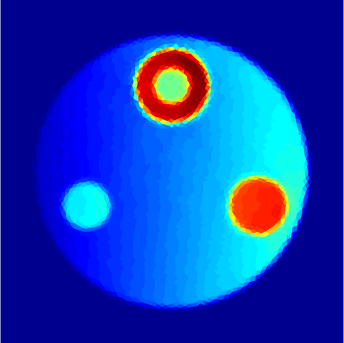}\hspace{\s}
 \input{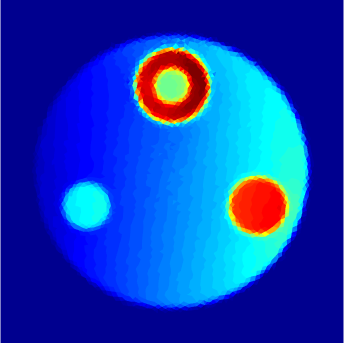}\hspace{\s}
 \input{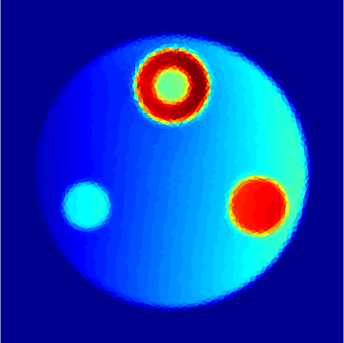} 
\caption{Reconstruction of the Lam\'e coefficients $(\lambda_2,\mu_2)$:   influence of the number of data $\{ {\g u}^{\ell}\}$;  
Lines: $\lambda$ and $\mu$; From left to right: Exact Lam\'e coefficients and  their reconstructions obtained respectively with
$1$, $2$ and $4$ sets of data $\{ {\g u}^{\ell}\}$. Here, we used $\e_{TV} = 10^{-4}$ and $\e_{\text{elas}} = 10^{-4}$.}
\label{fig:mulambda_2}
 \end{center}
 \end{figure}

\subsubsection{Anisotropic tensor inversion}

The last example concerns  the case of an anisotropic shear tensor  
$$ \g C =  \mu^{(1)} {\g C}^{1} + \mu^{(2)} {\g C}^{2} + \mu^{(3)} {\g C}^{3}.$$
The motivation is to show that our methodology can be adapted to any kind of model for the elasticity tensor $C$. 
Like previously,   we use $\e_{TV} = 10^{-4}$, $\e_{\text{elas}} = 10^{-4}$, and the reconstructions obtained 
with different number of data sets are plotted on each column of Figure  \ref{fig:anisotropic_1}. 

Notice that as in the case of an isotropic elastic medium, we successfully  reconstructed a quantitative approximation 
of the  scalar fields $\mu^{(k)}$ even in the case of one set of data.

\begin{figure}
\begin{center}
\def\imagescale{0.49}
\def\textscale{\textscaledef}
\def\s{\sdef}
 \input{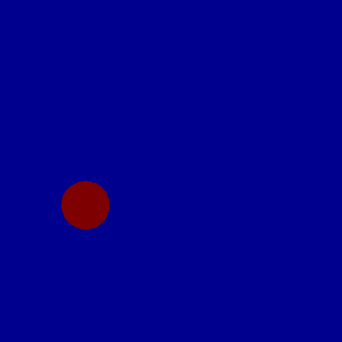}\hspace{\s}
 \input{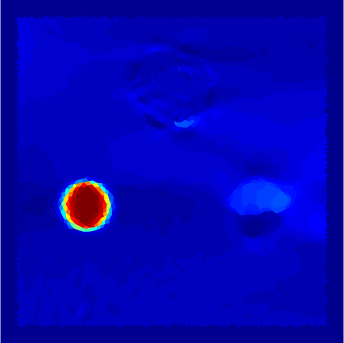}\hspace{\s}
 \input{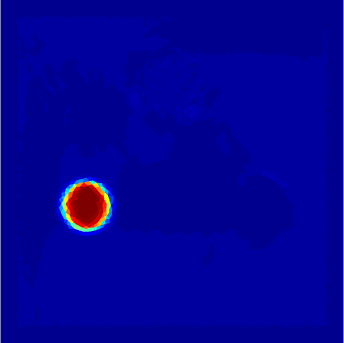}\hspace{\s}
 \input{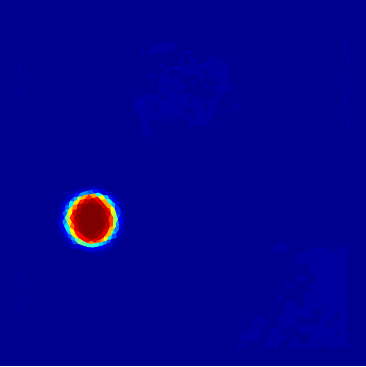} \\
 \input{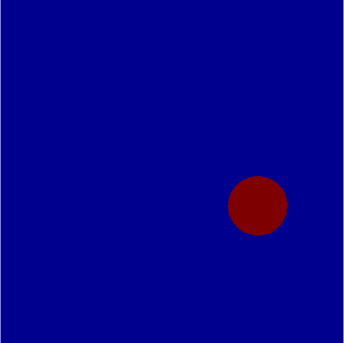}\hspace{\s}
 \input{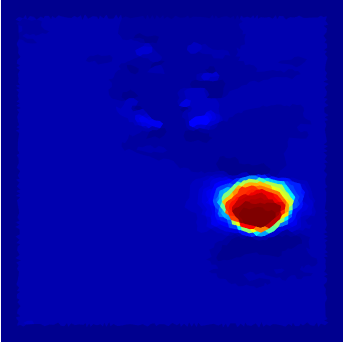}\hspace{\s}
 \input{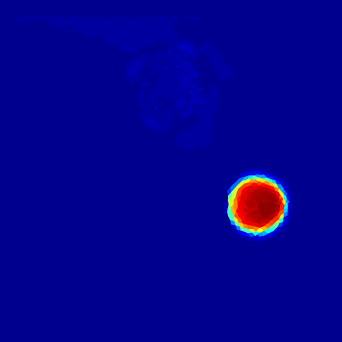}\hspace{\s}
 \input{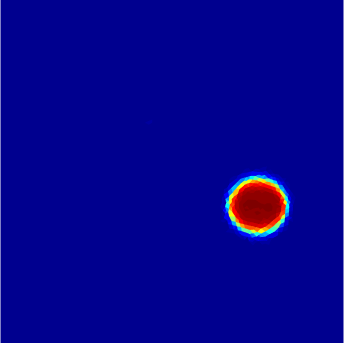} \\
 \input{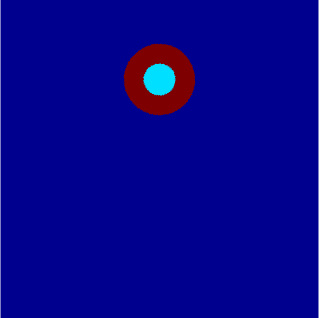}\hspace{\s}
 \input{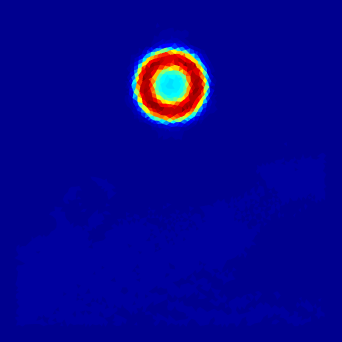}\hspace{\s}
 \input{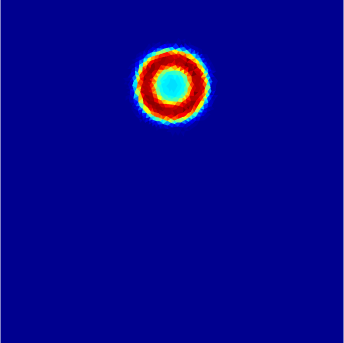}\hspace{\s}
 \input{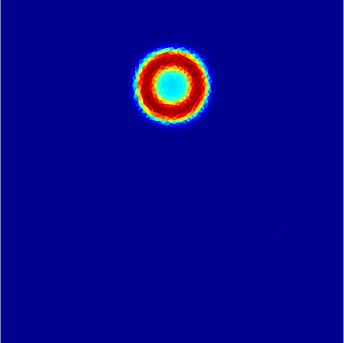} \\
\caption{Reconstruction of the anisotropic  coefficients $(\mu^{(1)},\mu^{(2)},\mu^{(3)})$:  
influence of the number of data $\{ {\g u}^{\ell}\}$ ;  
Lines: $\mu^{(1)}$, $\mu^{(1)}$ and $\mu^{(3)}$ ;
From left to right: Exact coefficients $\mu^{(k)}$ and their reconstructions obtained respectively with
$1$, $2$ and $4$ sets of data $\{ {\g u}^{\ell}\}$. Here, we used $\e_{TV} = 10^{-4}$ and $\e_{\text{elas}} = 10^{-4}$.}
\label{fig:anisotropic_1}
 \end{center}
 \end{figure}

\section{Concluding remarks} 
In this paper, we have introduced a new direct and stable method for reconstructing discontinuous elastic parameters from internal measurements of the displacement fields. We have proved an $L^2$-stability result with only one measurement. We have described a direct discretization of the inverse problem in both the isotropic and anisotropic cases. We have presented a variety of numerical results to illustrate the performance of our approach. In a forthcoming paper, we will apply our approach to real and clinical 
data using only measurements of one component of the displacement field and to  shear wave spectroscopy.   

\appendix

\section{Notations and tools}
\subsection{Tensor notations}

\begin{definition}\label{de:tensornotation} ~
We denote by $\R^{d\times d}$ the space of real matrices and $\R^{d\times d}_\sym$ the space of real symmetric matrices. Notice that $\R^{d\times d}_\sym\sim \R^{d(d+1)/2}$. We denote by $T^4= \R^{d^4}$ the space of order $4$ real tensors. We recall that

\begin{itemize}
\item[(i)] $A:B=\sum_{ij}A_{ij}B_{ij}\in\R$ for $A,B\in \R^{d\times d}$;
\item[(ii)]  $(A\otimes B)_{ijkl}=A_{ij}B_{kl}\in T^4$ for $A,B\in \R^{d\times d}$;
\item[(iii)]  $(\g A:B)_{ij}=\sum_{kl}\g A_{ijkl}B_{kl}\in\R^{d\times d}$ for $\g A\in T^4$ and $B\in \R^{d\times d}$; 
\item[(iv)]  $(B:\g A)_{ij}=\sum_{kl}B_{kl} \g A_{klij}\in\R^{d\times d}$ for $\g A\in T^4$ and $B\in \R^{d\times d}$; 
\item[(v)]  $(\g A:\g B)_{ijkl}=\sum_{mn}\g A_{ijmn}\g B_{mnkl}\in T^4$ for $\g A,\g B\in T^4$;
\item[(vi)]  $\g A|\g B=\sum_{ijkl}\g A_{ijkl}\g B_{ijkl}\in\R$ for $\g A,\g B\in T^4$.
\end{itemize}We define $T^4_\sym$ to be the space of all tensors $\g T$ such that for any symmetric matrix $S\in\R^{d\times d}_\sym$, the matrix $\g T:S$ is also symmetric and for any antisymmetric matrix $A$, we have $\g T:A=0$. Remark that in dimension two, $T^4_\sym\sim \R^6$ and in dimension 3,  $T^4_\sym\sim \R^{21}$.
\end{definition}

\subsection{Sobolev spaces}
 \begin{definition} For any Lipschitz  domain $\Omega\subset\R^d$, we define 
 \begin{align*}
 W^{1,p}(\Omega):= \left\{u \in L^p(\Omega),\ |\nabla u|\in L^2(\Omega) \right\}.
 \end{align*}
  We also define the following space: \begin{align*}
  H^1_0(\Omega,\R^d):=\left\{\g u \in L^2(\Omega,\R^{d}),\ |\nabla \g u|\in L^2(\Omega), \g u|_{\partial \Omega} = \g 0\right\},
 \end{align*}
 equipped with the norm:
 \begin{align*}
\norm{\g u}{H^1_0(\Omega)} := \norm{\nabla^s\g u}{L^2(\Omega)}, 
 \end{align*}
 where
 $\nabla^s\g u=(\nabla \g u+\nabla \g u^T)/2$.
\end{definition}
\begin{remark}The fact that this definition for the norm is correct is a direct consequence of Korn's inequality and Poincar\'e's inequality.
\end{remark}

\begin{proposition}{Properties of $W^{1,p}(\Omega)$:}~\label{lem:w1p} If $p>d$, the following results hold. 
\begin{itemize}
\item[(i)] $W^{1,p}(\Omega) \hookrightarrow L^{\infty}(\Omega)$;
\item[(ii)] If $u,v \in W^{1,p}(\Omega)$, then $uv \in W^{1,p}(\Omega)$;
\item[(iii)] If $u\in W^{1,p}(\Omega)$ and $\varphi \in H^1_0(\Omega)$, then $u\varphi \in H^1_0(\Omega)$;
\item[(iv)] $u\in W^{1,p}(\Omega)$, $f\in H^{-1}(\Omega)$ implies that $uf\in H^{-1}(\Omega)$ and $$\Vert uf\Vert_{H^{-1}}\leq C \Vert u \Vert_{W^{1,p}} \Vert f \Vert_{H^{-1}}$$ for some constant $C$ independent of $u$ and $f$.
\end{itemize}
\end{proposition}

 \begin{lemma}[$\nabla$ has a closed range in $\{\mu_0\}^\perp$]\label{lem_cr} Let $\Omega$ be a Lipschitz domain of $\R^d$ and $\mu_0\in L^\infty(\Omega)$ be such that $\mu_0\geq m\geq 0$. Then, there exists a constant $c>0$ such that

\begin{equation}\nonumber
\forall\mu\in\{\mu_0\}^\perp,\quad \norm{\mu}{L^2(\Omega)}\leq c\norm{\nabla \mu}{H^{-1}(\Omega)}.
\end{equation}
\end{lemma}
\begin{proof} Suppose that this is false. Take a sequence $(\mu_n)$ such that $\norm{\mu_n}{L^2(\Omega)}=1$ and $\norm{\nabla \mu_n}{H^{-1}(\Omega)}\to 0$. Up to an extraction $\mu_n\overset{L^2(\Omega)}\rightharpoonup\mu$ and $\int_\Omega\mu_n\mu_0\to \int_\Omega\mu\mu_0 = 0$. Moreover, $\nabla \mu = 0$ and so $\mu$ is constant. Then $\mu=0$. As the embedding $L^2(\Omega)\hookrightarrow H^{-1}(\Omega)$ is compact, we get that $\norm{\mu_n}{H^{-1}(\Omega)}\to 0$. Saying now that $$\norm{\mu_n}{L^2(\Omega)}^2=\norm{\mu_n}{H^{-1}(\Omega)}^2+\norm{\nabla \mu_n}{H^{-1}(\Omega)}^2,$$ we arrive at a contradiction.
\end{proof}

\bibliographystyle{plain}
\bibliography{biblio}

\end{document}